\documentclass{amsart}
\usepackage{mathtools, microtype, mathrsfs, xfrac, hyperref, amssymb, enumerate, xspace, stmaryrd}
\usepackage[utf8]{inputenc}
\usepackage{tikz-cd}
\usepackage{pdflscape}
\usepackage{bm}

\usepackage{todonotes}

\numberwithin{equation}{section}

\numberwithin{subsection}{section}

\allowdisplaybreaks[1]

\newtheorem*{namedtheorem}{\theoremname}
\newcommand{\theoremname}{testing}

\newtheorem{theorem}{Theorem}[section]
\newtheorem{proposition}[theorem]{Proposition}
\newtheorem{proposition-definition}[theorem]
{Proposition-Definition}
\newtheorem{corollary}[theorem]{Corollary}
\newtheorem{lemma}[theorem]{Lemma}
\newtheorem*{theorem*}{Theorem}

\theoremstyle{definition}
\newtheorem{definition}[theorem]{Definition}

\newtheorem{remark}[theorem]{Remark}

\newtheorem*{question*}{Question}

\theoremstyle{remark}

\renewcommand{\mathcal}{\mathscr}

 \newcommand\cB{\mathcal{B}}
\newcommand\cC{\mathcal{C}} 
 
\newcommand\cG{\mathcal{G}}

 \newcommand\cP{\mathcal{P}}
 
\newcommand\cS{\mathcal{S}} 
 
 \newcommand\cX{\mathcal{X}}

\newcommand\CC{\mathbb{C}} 
 \newcommand\FF{\mathbb{F}}

 \newcommand\PP{\mathbb{P}}
\newcommand\QQ{\mathbb{Q}} \newcommand\RR{\mathbb{R}}

 \newcommand\ZZ{\mathbb{Z}}

\newcommand\bC{\mathbf{C}}

 \newcommand\bL{\mathbf{L}}
 
 \newcommand\bP{\mathbf{P}}

 \newcommand\rR{\mathrm{R}}

\newcommand\fC{\mathfrak{C}}

\newcommand\fG{\mathfrak{G}} 
\newcommand\fH{\mathfrak{H}}

\newcommand\fK{\mathfrak{K}}

\newcommand\fN{\mathfrak{N}}

\newcommand\fQ{\mathfrak{Q}} 
 
\newcommand\fS{\mathfrak{S}}

\newcommand\arr{\ifinner\to\else\longrightarrow\fi}

\newcommand\arrto{\ifinner\mapsto\else\longmapsto\fi}

\renewcommand\H{\operatorname{H}}

\def\displaytimes_#1{\mathrel{\mathop{\times}\limits_{#1}}}

\def\displayotimes_#1{\mathrel{\mathop{\bigotimes}\limits_{#1}}}

\newcommand\aut{\operatorname{Aut}}

\newcommand\spec{\operatorname{Spec}}

\newcommand\id{\mathrm{id}}

\newcommand{\underaut}{\mathop{\underline{\mathrm{Aut}}}\nolimits}

\newlength{\ignora}

\renewcommand{\setminus}{\smallsetminus}

\newcommand{\GL}{\mathrm{GL}}
\newcommand{\SL}{\mathrm{SL}}
\newcommand{\PGL}{\mathrm{PGL}}
\newcommand{\PSL}{\mathrm{PSL}}

\newcommand{\gal}{\operatorname{Gal}}

\DeclareFontFamily{U}{mathx}{\hyphenchar\font45}
\DeclareFontShape{U}{mathx}{m}{n}{
      <5> <6> <7> <8> <9> <10>
      <10.95> <12> <14.4> <17.28> <20.74> <24.88>
      mathx10
      }{}
\DeclareSymbolFont{mathx}{U}{mathx}{m}{n}
\DeclareFontSubstitution{U}{mathx}{m}{n}
\DeclareMathAccent{\widecheck}{0}{mathx}{"71}
\DeclareMathAccent{\wideparen}{0}{mathx}{"75}

\renewcommand{\epsilon}{\varepsilon}

\newcommand{\cha}{\operatorname{char}}

\newcommand{\diag}{\operatorname{diag}}

\begin{document}

\title{The field of moduli of plane curves}

\author{Giulio Bresciani}

\begin{abstract}
	We prove that a smooth, complex plane curve of odd degree can be defined by a polynomial with coefficients in $\mathbb{R}$ if and only if it is isomorphic to its complex conjugate; there are counterexamples in even degree. Over arbitrary base fields of characteristic $0$, we prove that a smooth plane curve of degree prime with $6$ can be defined by a polynomial with coefficients in the field of moduli. 
	
	We also prove results about fields of moduli of algebraic cycles in $\mathbb{P}^{2}$. In particular, these apply to singular plane curves of arbitrary degree, too.
\end{abstract}

\address[G. Bresciani]{Dipartimento di Matematica, Università di Pisa, Italy}
\email{giulio.bresciani1@unipi.it}


\maketitle

\tableofcontents

\section{Introduction}

We work over a field $k$ of characteristic $0$ with algebraic closure $K$. Let $C\subset\PP^{2}_{K}$ be a smooth plane curve.

\begin{question*}
	Does it exist a minimal extension $k'\subset K$ of $k$ such that $C$ is defined by a homogeneous polynomial with coefficients in $k'$?
\end{question*}

More generally, we say that a subfield $k'\subset K$ is a field of definition if $C$ has a model over $k'$, and we can ask whether a minimal field of definition exists.

Despite being a natural, easily stated question, little is known about it, with clear results available only for curves up to degree $4$.

\subsection*{Fields of moduli}

There is a good candidate for the minimal field of definition: the so-called \emph{field of moduli} of the curve. Let $H\subset\gal(K/k)$ be the subgroup of elements $\sigma\in\gal(K/k)$ such that $\sigma^{*}C$ is isomorphic to $C$ over $K$. The field of moduli $k_{C}$ of $C$ is the subfield of $K$ fixed by $H$; it coincides with the residue field of the corresponding point of the coarse moduli space. Furthermore, it coincides with the field of moduli of the embedded curve, i.e. of the pair $(\PP^{2}_{K},C)$, see Lemma~\ref{lemma:plane}.

The intersection of the fields of definition is almost always the field of moduli (e.g. if $k=\QQ$), and when this fails the intersection is $k_{C}(i)$, see Proposition~\ref{proposition:fieldint}. The question above is then essentially equivalent to the following.

\begin{question*}
	Is $C$ defined by a homogeneous polynomial with coefficients in $k_{C}$?
\end{question*}

\subsection*{Known results}

If $C$ has degree $3$, i.e. it is an elliptic curve, it is a classical fact rooted in the work of F. Klein \cite{klein} that $k_{C}$ is generated over $k$ by the $j$-invariant and that $C$ can be defined by a homogeneous polynomial with coefficients in $k(j)$. This fact is a cornerstone for much of the theory about elliptic curves.

M. Artebani and S. Quispe \cite[Theorem 0.2]{artebani-quispe} proved that, if $k=\RR$ and $K=\CC$, a smooth plane quartic $C$ over $\CC$ with $\aut(C)\neq C_{2}$ is defined over the field of moduli, and have given a counterexample with $\aut(C)=C_{2}$. Much less is known in higher degree. Before the present article, it was not known whether there existed a $d>4$ such that all curves of degree $d$ are defined over the field of moduli.

B. Huggins proved that the automorphism group of a plane curve not defined over the field of moduli is either diagonal or one of two non-abelian groups \cite[Theorem 6.4.8]{hugginsphd}. E. Badr also proved that some complex plane curves with field of moduli $\RR$ such that $\aut(\PP^{2}_{\CC},C)$ satisfies a technical condition are defined over $\RR$ \cite{badr20}. Various types of examples of plane curves not defined over the field of moduli can be found in \cite{hugginsphd} \cite{artebani-quispe} \cite{badr-bars-garcia} \cite{badr-bars}. We are not aware of other results.

\subsection*{New results}

We prove that plane curves of degree prime with $6$ are defined over the field of moduli, and that odd degree is sufficient if $k=\RR$.

\begin{theorem}\label{theorem:no6}
	A smooth plane curve over $K$ of degree prime with $6$ and with field of moduli $k$ is defined by a homogeneous polynomial with coefficients in $k$.
\end{theorem}

\begin{theorem}\label{theorem:realcomplex}
	Let $C$ be a smooth complex plane curve of odd degree and $\bar{C}$ its complex conjugate. The following are equivalent.
	\begin{enumerate}[(i)]
		\item There exists a homogeneous polynomial with real coefficients $p\in\RR[x,y,z]$ defining $C$.
		\item There exists a homogeneous polynomial $p\in\CC[x,y,z]$ for $C$ and a linear transformation $g\in\GL_{3}(\CC)$ such that $p\circ g=\bar{p}$.  
		\item The curves $C$ and $\bar{C}$ are isomorphic as abstract curves.
	\end{enumerate}
\end{theorem}

We find it surprising that Theorem~\ref{theorem:realcomplex} has gone unnoticed until now, since its statement is completely elementary and plane complex curves have been studied for centuries. Still, we couldn't find it in the literature. Counterexamples are known in even degree \cite[\S 7.1]{hugginsphd}.

We also prove that most plane curves are defined over the field of moduli even if the degree is not prime with $6$: only a small fraction of the groups considered by Huggins in \cite[Theorem 6.4.8]{hugginsphd} can actually appear as automorphism groups of plane curves not defined over the field of moduli, and this fraction is even smaller for curves whose degree is either odd or prime with $3$, see Theorems~\ref{theorem:no3}, \ref{theorem:odd}, \ref{theorem:sextics}. In particular, in Corollary~\ref{corollary:quartics} we prove that if $C$ is a quartic and $|\aut(C)|\neq 2$ then $C$ is defined by a homogeneous polynomial with coefficients in the field of moduli, thus generalizing to arbitrary base fields of characteristic $0$ the result of Artebani and Quispe \cite{artebani-quispe}.

\subsection*{Algebraic cycles}

Theorems~\ref{theorem:no6} and \ref{theorem:realcomplex} are direct consequences of the following Theorem~\ref{theorem:six}, which holds for arbitrary algebraic cycles in $\PP^{2}_{K}$ with finite automorphism group. In particular, it holds singular plane curves of arbitrary degree.

Given an algebraic cycle $Z$ on $\PP^{2}_{K}$, consider the subgroup $H\subset\gal(K/k)$ of Galois automorphisms $\sigma$ such that $\sigma^{*}Z$ is linearly equivalent to $Z$, i.e. there exists $g\in\PGL_{3}(K)$ such that $g(\sigma^{*}Z)=Z$; the field of moduli $k_{Z}$ of $Z$ is the fixed field $K^{H}$. If $Z=C$ is a smooth, plane curve of degree $\ge 4$, it coincides with the field of moduli $k_{C}$ of the abstract curve, see Lemma~\ref{lemma:plane}. Denote by $\aut_{K}(\PP^{2}_{K},Z)$ the subgroup of $\PGL_{3}(K)$ of projective linear transformations mapping $Z$ to itself.

\begin{theorem}\label{theorem:six}
	Let $Z$ be an algebraic cycle on $\PP^{2}_{K}$ such that $\aut_{K}(\PP^{2}_{K},Z)$ is finite, write $k_{Z}$ for the field of moduli. There exists a finite extension $k'/k_{Z}$ with $[k':k_{Z}]\le 3$ such that $Z$ descends to a cycle in $\PP^{2}_{k'}$. If $Z=C$ is a smooth plane curve, we can choose $k'$ so that $[k':k_{Z}]=[k':k_{C}]\mid\deg C$.
\end{theorem}

Theorem~\ref{theorem:six} is in turn based on the following Theorem~\ref{theorem:cycles}, which is an analysis of the possible automorphism groups of algebraic cycles not defined over the field of moduli. In order to state this other result, let us give a definition. Write $\diag(\alpha,\beta,\gamma)$ for the $3\times 3$ diagonal matrix with eigenvalues $\alpha,\beta,\gamma$, and $\zeta_{n}=e^{\frac{2\pi i}{n}}\in\bar{\QQ}$.

\begin{definition}
	Let $k$ be a field of characteristic $0$ with algebraic closure $K$. A finite subgroup $G\subset\PGL_{3}(K)$ is \emph{critical} if it is conjugate to
	\begin{itemize}
		\item the abelian subgroup $C_{a}\times C_{an}$ generated by $\diag(\zeta_{a},1,1)$, $\diag(1,\zeta_{a},1)$, $\diag(\zeta_{an},\zeta_{an}^{d},1)$ for positive integers $a,n,d$ satisfying $d^{2}-d+1\cong 0\pmod{n}$ and $3\mid an$,
		\item the abelian subgroup $C_{a}\times C_{a2^{b}n}$ generated by $\diag(\zeta_{a},1,1)$, $\diag(1,\zeta_{a},1)$, $\diag(\zeta_{a2^{b}n},\zeta_{a2^{b}n}^{d},1)$ for some positive integers $a,b,n,d$ with $d^{2}\cong 1\pmod{n}$, $d\cong\pm1\pmod{2^{b}}$, and $n$ odd.
		\item the Hessian subgroup $H_{2}\simeq C_{3}^{2}\rtimes C_{2}$ of order $18$ (see \S\ref{sect:hessian}).
		\item if $\zeta_{12}\not\in k$, the Hessian subgroup $H_{3}\simeq C_{3}^{2}\rtimes C_{4}$ of order $36$ (see \S\ref{sect:hessian}).
	\end{itemize}
	
	We say that $G$ is \emph{lucky} if it is not conjugate to one of the groups above, and furthermore it is not conjugate to
	\begin{itemize} 
		\item the abelian subgroup $C_{a}\times C_{an}$ generated by $\diag(\zeta_{a},1,1)$, $\diag(1,\zeta_{a},1)$, $\diag(\zeta_{an},\zeta_{an}^{d},1)$ for positive integers $a,n,d$ satisfying $d^{2}-d+1\cong 0\pmod{n}$,
		\item the Hessian subgroup $H_{1}\simeq C_{3}^{2}$ of order $9$ (see \S\ref{sect:hessian}).
	\end{itemize}
	In particular, if $\zeta_{12}\in k$ and $G$ is conjugate to $H_{3}$, then $G$ is lucky.
	
	As we will see, a subgroup of $\PGL_{3}$ is critical if it might give an obstruction to descent, while it is lucky if descent always works and as a bonus we may descend to $\PP^{2}$, as opposed to a Brauer-Severi surface. 
\end{definition}

\begin{remark}
	The large majority of finite subgroups of $\PGL_{3}$ are lucky, with the noteworthy exception of the trivial group which is not lucky nor critical; see \S\ref{sect:subgroups} for a complete list of finite subgroups of $\PGL_{3}$. In particular, only two non-abelian subgroups are not lucky, one if $\zeta_{12}\in k$, and the large majority of the abelian ones are lucky too.
\end{remark}

\begin{theorem}\label{theorem:cycles}
	Let $Z$ be a cycle on $\PP^{2}_{K}$ with $\aut(\PP^{2}_{K},Z)\subset\PGL_{3}(K)$ finite and with field of moduli equal to $k$. If $\aut(\PP^{2}_{K},Z)$ is not critical, then $Z$ descends to a cycle on some Brauer-Severi surface over $k$. If $G$ is lucky, $Z$ descends to $\PP^{2}_{k}$.
	
	On the other hand, if $G\subset\PGL_{3}(\bar{\QQ})$ is critical there exists a field $k$ of characteristic $0$ with algebraic closure $K$ and a cycle $Z$ on $\PP^{2}_{K}$ not defined over its field of moduli with $\aut(\PP^{2}_{K},Z)=G$. If $G$ is not conjugate to $H_{3}$, we may choose $k$ so that it contains $\CC$.
\end{theorem}

Theorem~\ref{theorem:cycles} can be viewed as a generalization and sharpening of Huggins' theorem \cite[Theorem 6.4.8]{hugginsphd}. We remark that Huggins' version of the theorem is not sharp enough to obtain Theorems~\ref{theorem:no6}, \ref{theorem:realcomplex} and \ref{theorem:six}.

Theorems~\ref{theorem:six} and \ref{theorem:cycles} hold in even greater generality for \emph{algebraic structures} in the sense of \cite[\S 5]{giulio-angelo-moduli}, see Theorems~\ref{theorem:structures} and \ref{theorem:critical}. In \cite{giulio-points}, we study in detail the case of finite subsets of $\PP^{2}$.

\subsection*{A remark about our techniques} 

In our proofs, we heavily use the theory algebraic stacks: this is a novel approach for studying fields of moduli and fields of definition. We have laid the foundations of this approach in a recent joint work with A. Vistoli \cite{giulio-angelo-moduli}; the present article is the first instance in which our methods are applied in an essential way to an open problem. 

One of the main strategies we use is based on finding rational points on a certain smooth Deligne-Mumford stack $\cX$ with $\dim \cX=\dim X$, where $X$ is the variety we are studying (in our case, $X=\PP^{2}$). If $\dim X=1$, then $\cX$ has a rational point if and only if its coarse moduli space has a rational point. This fact is implicitly used in a result by P. Dèbes and M. Emsalem \cite[Corollary 4.3.c]{debes-emsalem} which is the backbone of a lot of the literature about fields of moduli of curves: the coarse moduli space of $\cX$ is what they call \emph{the canonical model of $X/\aut(X)$}.

If $\dim X\ge 2$, e.g. if $X=\PP^{2}$, this is simply wrong: the coarse moduli space of $\cX$ might have a rational point even though $\cX(k)=\emptyset$. The reason behind this is the fact that quotients of smooth varieties are not necessarily smooth in dimension $\ge 2$: if we pass to the coarse moduli space of $\cX$, we lose smoothness. Because of this, stacks are especially helpful when studying fields of moduli in dimension $\ge 2$.

\subsection*{Acknowledgements}

I would like to thank A. Vistoli for pointing out to me the fact that every automorphism of a smooth plane curve of degree $\ge 4$ is linear.

\subsection*{Notations and conventions}

We work over a field $k$ of characteristic $0$ with algebraic closure $K$. We fix a basis of $k^{3}$ so that it makes sense to talk about diagonal and permutation matrices (recall that a permutation matrix is a matrix which permutes the basis), and we fix a preferred embedding $\GL_{2}\subset\PGL_{3}$ using the first two coordinates.

With an abuse of terminology, we often identify matrices with their images in $\PGL_{3}(K)$. We write $\diag(a_{1},a_{2},a_{3})\in\PGL_{3}(K)$ for the image in $\PGL_{3}(K)$ of the diagonal $3\times 3$ matrix with eigenvalues $a_{1},a_{2},a_{3}$, so $\diag(a_{1},a_{2},a_{3})=\diag(\lambda a_{1},\lambda a_{2},\lambda a_{3})$ for $\lambda\neq 0$. We say that a subgroup of $\PGL_{3}(K)$ is diagonal if its elements are diagonal.

Let $G$ be a group acting on some space or set $X$, $Z\subset X$ a subspace, $g\in G$ an element. We say that $g$ \emph{stabilizes} $Z$, or that $Z$ is $g$-invariant, if $g(Z)=Z$. We say that $g$ \emph{fixes} $Z$ if $g$ restricts to the identity on $Z$. We say that $G$ stabilizes (resp. fixes) $Z$ if every element $g\in G$ stabilizes (resp. fixes) $Z$. The \emph{fixed locus} of $g$ (resp. $G$) is the subspace of points $x\in X$ with $gx=x$ (resp. $\forall g\in G:gx=x$).

\section{The two main strategies}\label{sect:strategies}

We recall some basic constructions from \cite{giulio-angelo-moduli} and \cite{giulio-fmod}. Let $k$ be a field of characteristic $0$ with algebraic closure $K$ and $\xi$ an algebraic structure, e.g. a cycle, on $\PP^{2}_{K}$ with field of moduli $k$ in the sense of \cite[\S 5]{giulio-angelo-moduli}, write $G=\aut(\PP^{2},\xi)\subset\PGL_{3}(K)$.

There is a finite gerbe $\cG_{\xi}$ over $k$, called the \emph{residual gerbe}, with a universal projective bundle $\cP_{\xi}\to\cG_{\xi}$ of relative dimension $2$ which is a twisted form over $k$ of the the natural morphism $[\PP^{2}_{K}/G]\to\cB_{K}G$. The residual gerbe $\cG_{\xi}$ and the projective bundle $\cP_{\xi}\to\cG_{\xi}$ are characterized by the following property: given a scheme $S$ over $k$, a morphism $S\to\cG_{\xi}$ corresponds to a twisted form of $\xi$ on the projective bundle $\cP_{\xi}|_{S}\to S$. In particular, $\xi$ descends to a structure on some Brauer-Severi surface over $k$ if and only if $\cG_{\xi}(k)\neq\emptyset$, and it descends to a structure on $\PP^{2}_{k}$ if and only if $\cP_{\xi}(k)\neq\emptyset$.

Let us break down the definition in the case of cycles. If $C\subset\PP^{2}_{K}$ is a reduced, irreducible closed subscheme and $S$ is a scheme over $k$, a twisted form of $C$ over $S$ is a projective bundle $\cP\to S$ with a closed subscheme $\cC\subset \cP$ such that there exists a finite subextension $K/k'/k$ and a scheme $S'$ over $k'$ with an étale covering $S'\to S$ such that $(\cP|_{S'}\times_{k'}K,\cC|_{S'}\times_{k'}K)\simeq (\PP^{2}_{K}\times_{K} S'_{K},C\times_{K} S'_{K})$. As expected, $C\subset\PP^{2}_{K}$ defines a twisted form of $C\subset\PP^{2}_{K}$: this follows from the fact that $C$ descends to some finite subextension $K/k'/k$ and hence we may take $S'=S=\spec K$ in the definition above. 

If $Z=\sum_{i}n_{i}C_{i}$ is a cycle, a twisted form of $Z$ over $S$ is a projective bundle $\cP\to S$ and a formal sum $\sum_{i}n_{i}\cC_{i}$ where $\cC_{i}\subset\cP$ is a twist of $C_{i}$. The residual gerbe $\cG_{Z}$ is the functor $S\mapsto$\{twisted forms of $Z$\}; if $\aut(\PP^{2},Z)$ is finite then $\cG_{Z}$ is a Deligne-Mumford stack which is a gerbe and $\cP_{Z}\to\cG_{Z}$ is the corresponding universal bundle given by Yoneda's lemma. The trivial twist of $Z$ defines a tautological morphism $\spec K\to\cG_{Z}$, and the coarse moduli space of $\cG_{Z}$ is the spectrum of the field of moduli of $Z$.

\subsection{Showing that $\xi$ is defined over $\PP^{2}_{k}$} The obvious question is: ho do we find rational points on $\cP_{\xi}$? Let $\bP_{\xi}$ be the coarse moduli space of $\cP_{\xi}$, it is called the \emph{compression} of $\xi$. Since $\cP_{\xi,K}=[\PP^{2}_{K}/G]$, then $\bP_{\xi,K}=\PP^{2}_{K}/G$. Since the action of $G$ on $\PP^{2}_{K}$ is faithful, the natural morphism $\cP_{\xi}\to\bP_{\xi}$ is birational, hence we have a rational map $\bP_{\xi}\dashrightarrow\cP_{\xi}$. Suppose that we find a rational point $p\in\bP_{\xi}(k)$ which lifts to a rational point of a resolution of singularities of $\bP_{\xi}$ (by the Lang--Nishimura theorem, this condition does not depend on the resolution). The Lang--Nishimura theorem for tame stacks \cite[Theorem 4.1]{giulio-angelo-valuative} then implies that $\cP_{\xi}(k)\neq\emptyset$.

So we want to find rational points on the compression $\bP_{\xi}$. A closed subspace $D\subset\PP^{2}_{K}$ is \emph{distinguished} \cite[Definition 17]{giulio-fmod} if $\tau(D)=D$ for every $\tau\in\PGL_{3}(K)$ such that $\tau^{-1}G\tau=G$; in particular, $D$ is $G$-invariant. If $D$ is distinguished, then $D/G\subset\PP^{2}/G$ descends to a closed subset of $\bP_{\xi}$ \cite[Lemma 18]{giulio-fmod}. In \cite[Example 19]{giulio-fmod} we give some strategies to construct distinguished subsets. Assume that the action of $G$ on $D$ is transitive, then $D/G$ descends to a rational point $p\in\bP_{\xi}(k)$.

A rational point of a variety is \emph{liftable} \cite[Definition 6.6]{giulio-angelo-moduli} if it lifts to a resolution of singularities. A singularity is \emph{of type $\rR$} if every twisted form of it is liftable. To show that $p$ is liftable, it is enough to check that its singularity is of type $\rR$ using the classification given in \cite{giulio-tqs2}. Let $d\in D$ be any point, $G_{d}\in G$ the stabilizer, the singularity of $\PP^{2}_{K}/G=\bP_{\xi,K}$ in the image of $d$ is equivalent \cite[\S 6.2]{giulio-angelo-moduli} to the one of $T_{d}\PP^{2}/G_{d}$ in the image of the origin, hence it is enough to show that $T_{d}\PP^{2}/G_{d}$ is of type $\rR$. Finally, a group is of type $\rR_{2}$ if the quotient of every faithful $2$-dimensional representation is of type $\rR$, hence sometimes it is sufficient to check that $G_{z}$ is $\rR_{2}$ using the classification given in \cite{giulio-tqs2}.

\subsection{Finding $\xi$ so that it is not defined over $k$}

Fix $G\subset\PGL_{3}(\bar{\QQ})$ finite. Suppose that we want to find a field $k$ and a structure $\xi$ on $\PP^{2}_{K}$ with $G=\aut(\PP^{2},\xi)\subset\PGL_{3}(K)$ such that the field of moduli of $\xi$ is $k$ but $\xi$ is not defined over $k$. First, we want to make sure that $G$ descends to a group scheme $\fG\subset\PGL_{3,k}$ acting on $\PP^{2}_{k}$, for instance we might restrict ourselves to fields containing enough elements to define each matrix of $G$.

If we find $k$ and a subgroup $\fN\subset\PGL_{3,k}$ containing and normalizing $\fG$ so that $\H^{1}(k,\fN)\to\H^{1}(k,\fN/\fG)$ is not surjective, then by \cite[Theorem 4]{giulio-fmod} there exists a structure $\xi$ on $\PP^{2}_{K}$ with field of moduli $k$ which is not defined over $k$. By \cite[Theorem 3]{giulio-fmod}, $\xi$ can be interpreted as the structure of some $0$-cycle.

\section{Finite subgroups of $\PGL_{3}$}\label{sect:subgroups}

Since we are in characteristic $0$, the finite subgroups of $\PGL_{3}(\bar{\QQ}), \PGL_{3}(K)$ and $\PGL_{3}(\CC)$ coincide, and they are completely classified, see \cite[Chapter XII]{mbd}. Before giving the list, for the convenience of the reader we recall some facts which will later play an important role.

\subsection{Abelian subgroups}

\begin{lemma}
	Let $\alpha\in K\setminus\{0,1\}$ be an element. The centralizer of $\diag(\alpha,\alpha,1)$ is $\GL_{2}(K)\subset\PGL_{3}(K)$.
\end{lemma}

\begin{proof}
	Let $g$ be an element of the centralizer, then $g$ must stabilize the fixed locus of $\diag(\alpha,\alpha,1)$, namely the point $(0:0:1)$ and the line $\{(s:t:0)\}$. The statement follows.
\end{proof}

\begin{lemma}\label{lemma:centralizer}
	Let $\alpha\neq\beta\in K\setminus\{0,1\}$ be different elements. If $\{\alpha,\beta\}=\{\zeta_{3},\zeta_{3}^{2}\}$, the centralizer of $\diag(\alpha,\beta,1)$ in $\PGL_{3}(K)$ is generated by the diagonal matrices and by a permutation matrix of order $3$, otherwise it is the group of diagonal matrices.
\end{lemma}

\begin{proof}
	Assume that $g\in\PGL_{3}(K)$ is in the centralizer, then $g$ must stabilize the fixed locus of $\diag(\alpha,\beta,1)$, i.e. the three points $(0:0:1),(0:1:0),(1:0:0)$. Up to a diagonal matrix, we may then assume that $g$ is a permutation matrix. If $g$ fixes the three points, then it is the identity. If $g$ acts as a transposition, then it is immediate to check that $g$ does not commute with $\diag(\alpha,\beta,1)$ since $\alpha,\beta,1$ are pairwise different. If $g$ acts as a $3$-cycle, then $\diag(\alpha,\beta,1)=\diag(1,\alpha,\beta)$ which implies $\{\alpha,\beta\}=\{\zeta_{3},\zeta_{3}^{2}\}$.
\end{proof}

Write $H_{1}$ for the group isomorphic to $C_{3}^{2}$ generated by $\diag(\zeta_{3},\zeta_{3}^{2},1)$ and by a permutation matrix of order $3$. The group $H_{1}$ is not diagonalizable, since it has no fixed points.

\begin{corollary}\label{corollary:abelian}
	A finite, abelian subgroup of $\PGL_{3}(K)$ is either conjugate to $H_{1}$ or diagonalizable.
\end{corollary}

\subsection{The Hessian groups}\label{sect:hessian}

Consider the six matrices
\[ M_{0}=\left( \begin{array}{ccc}
1 & 0 &	0 \\
0 & \zeta_{3} & 0 \\
0 & 0 & \zeta_{3}^{2} \\
\end{array} \right),~
M_{1}=\left( \begin{array}{ccc}
0 & 0 & 1 \\
1 & 0 & 0 \\
0 & 1 & 0
\end{array} \right),~
M_{2}=\left( \begin{array}{ccc}
1 & 0 & 0 \\
0 & 0 & 1 \\
0 & 1 & 0
\end{array} \right),
\]

\[
M_{3}=\left( \begin{array}{ccc}
1 & 1 & 1 \\
1 & \zeta_{3} & \zeta_{3}^{2} \\
1 & \zeta_{3}^{2} & \zeta_{3}
\end{array} \right),{}~
M_{4}=\left( \begin{array}{ccc}
1 & 1 & \zeta_{3} \\
1 & \zeta_{3} & 1 \\
\zeta_{3}^{2} & \zeta_{3} & \zeta_{3}
\end{array} \right),
M_{5}=\left( \begin{array}{ccc}
1 & 0 &	0 \\
0 & 1 & 0 \\
0 & 0 & \zeta_{3} \\
\end{array} \right).\]

For $i=1,\dots,5$, let $H_{i}\subset\PGL_{3}(\QQ(\zeta_{3}))$ be the group generated by $M_{0},\dots, M_{i}$ for $i=1,\dots,5$. We call the subgroups $H_{1},\dots,H_{5}$ the \emph{Hessian groups} \cite[Chapter XII]{mbd} of degrees $9,18,36,72,216$ respectively. In the literature, only $H_{5}$ is consistently called Hessian, while only some authors call Hessian the others. Notice that our matrices differ by a scalar from the ones given in \cite[Chapter XII]{mbd}: since we work in $\PGL_{3}(K)$, as opposed to $\GL_{3}(K)$, we can forget about some scalars and simplify everything.

The group $H_{3}$ is not normal in $H_{5}$, but in all the other cases $H_{i}$ is normal in $H_{j}$ for $j>i$. We have isomorphisms
\[H_{1}\simeq C_{3}^{2},~H_{2}\simeq (C_{3}^{2})\rtimes C_{2},~H_{3}\simeq (C_{3}^{2})\rtimes C_{4},\]
where the action of $C_{4}=\left<M_{3}\right>$ on $C_{3}^{2}$ is given by the matrix $\begin{psmallmatrix}  & -1  \\ 1 &  \end{psmallmatrix}$ in $\SL(2,3)$, and $C_{2}\subset C_{4}$ acts as $-1$. Furthermore, we have isomorphisms
\[H_{5}/H_{1}\simeq\SL(2,3),~H_{4}/H_{1}\simeq Q_{8}\subset \SL(2,3),\]
\[H_{5}/H_{2}\simeq\PSL(2,3)\simeq A_{4},~H_{4}/H_{2}=\operatorname{Kl}\subset A_{4}\]
where $Q_{8}$ is the quaternion group and $\operatorname{Kl}\subset A_{4}$ is the Klein group. 

The fixed subset of any non-trivial cyclic subgroup of $H_{1}\simeq C_{3}^{2}$ consists of three non-collinear points; these four triangles are pairwise disjoint. Any line connecting two points of two different triangles contains exactly one point of each triangle: in particular, the union of one triangle and one point of another triangle is in general position.

Since these four triangles correspond to the cyclic subgroups of $H_{1}\simeq\FF_{3}^{2}$, it is natural to identify them with the four points of $\PP(\FF_{3}^{2})$. The group $\SL(2,3)$ is a non-split central extension of $\PSL(2,3)\simeq A_{4}$ by $C_{2}=\left<-1\right>$, and the induced action of $\PSL(2,3)\simeq A_{4}$ on $\PP(\FF_{3}^{2})$ is the standard one of $A_{4}$ on four points.

\begin{lemma}\label{lemma:H1norm}
	The normalizer of $H_{1}\subset\PGL_{3}(K)$ is $H_{5}$.
\end{lemma}

\begin{proof}
	Let $g\in\PGL_{3}(K)$ be an element normalizing $H_{1}$, in particular $g$ acts on the set of four triangles $\PP^{2}(\FF_{3}^{2})$. Since $H_{5}$ acts as $A_{4}$ on $\PP^{2}(\FF_{3}^{2})$, up to mutliplying $g$ by an element of $H_{5}$ we may assume that $g$ acts trivially on $\PP^{2}(\FF_{3}^{2})$, in particular it stabilizes the two triangles of fixed points of $M_{0}$ and $M_{1}$. Furthermore, up to multiplying by an element of $H_{2}$ we may assume that the points $(0:0:1)$, $(0:1:0)$, $(1:0:0)$ are fixed by $g$, hence $g$ is diagonal. Since an element of $\PGL_{3}(K)$ fixing four points in general position is trivial, there are at most three elements of $\PGL_{3}(K)$ fixing $(0:0:1)$, $(0:1:0)$, $(1:0:0)$ and permuting the fixed locus of $M_{1}$. The three powers of $M_{0}$ do this, so $g$ is one of them.	
\end{proof}

\subsection{The Hessian group schemes}\label{sect:hessch}

Denote by $\bar{M}_{i}$ the Galois conjugate of $M_{i}$ with respect to the only non-trivial automorphisms of $\QQ(\zeta_{3})/\QQ$, we have identities in $\PGL_{3}(\QQ(\zeta_{3}))$
\[\bar{M}_{0}=M_{0}^{-1},~\bar{M}_{1}=M_{1},~\bar{M}_{2}=M_{2},\]
\[\bar{M}_{3}=M_{3}^{-1},~ \bar{M}_{4}=M_{4}\cdot M_{3},~\bar{M}_{5}=M_{5}^{-1}.\]
In particular, the Galois action stabilizes each $H_{i}$, hence for every $i=1,\dots,5$ we get a finite subgroup scheme $\fH_{i}\subset\PGL_{3,\QQ}$ with $\fH_{i}(\QQ(\zeta_{3}))=H_{i}$. Write 
\[\fC_{4}=\fH_{3}/\fH_{1},\fQ_{8}=\fH_{4}/\fH_{1},\fK=\fH_{4}/\fH_{2},\]
they are twisted forms over $\QQ$ of $C_{4},Q_{8},\operatorname{Kl}$ respectively and there is a short exact sequence
\[1\to \fC_{4}\to \fQ_{8}\to \fK\to 1.\]

Write $\{\pm 1,\pm i,\pm j,\pm k\}$ for the elements of $Q_{8}$, we may assume that $M_{3},M_{4}$ map respectively to $i,j$ in $Q_{8}=H_{4}/H_{1}$ and that $C_{4}=\left<i\right>\subset Q_{8}$. Using the characterizations of $\bar{M}_{i}$ written above, the non-trivial element of $\gal(\QQ(\zeta_{3})/\QQ)$ acts on $C_{4}\subset Q_{8}$ as $i\mapsto -i$, $j\mapsto -k$, $k\mapsto -j$ and on $\operatorname{Kl}=C_{2}^{2}$ as $(1,0)\mapsto (0,1)$, $(0,1)\mapsto (1,0)$.

There is an induced action of $\fQ_{8}$ on $\fH_{1}=\mu_{3}\times C_{3}$, and we have isomorphisms
\[\fH_{1}\simeq \mu_{3}\times C_{3},~\fH_{2}\simeq (\mu_{3}\times C_{3})\rtimes C_{2},\]
\[\fH_{3}\simeq (\mu_{3}\times C_{3})\rtimes \fC_{4},~\fH_{4}/\fH_{1}\simeq \fQ_{8},~\fH_{4}/\fH_{2}=\fK.\]

\subsection{The list}

Here is the list of all finite subgroups of $\PGL_{3}(K)$ up to conjugation \cite[Chapter XII]{mbd}.

\begin{description}
	\item[(A)] Finite subgroups of $\GL_{2}(K)\subset\PGL_{3}(K)$.
	\item[(B)] A group generated by a non-trivial finite diagonal subgroup and by a permutation matrix of order $3$. A group of this type is either conjugate to $H_{1}$ or it has exactly one invariant triangle.
	\item[(C)] A group generated by a group of type {\bf (B)} and a matrix of the form $\begin{psmallmatrix}  & \beta &  \\ \alpha &  &  \\  &  & 1 \end{psmallmatrix}$. A group of this type is either conjugate to $H_{2}$ or it has exactly one invariant triangle.
	\item[(D)] The groups $H_{3}\subset H_{4}\subset H_{5}$. 
	\item[(E)] The simple groups $A_{5}$, $A_{6}$ and $\PSL(2,7)$.
\end{description}

\section{On the automorphism groups of structures on $\PP^{2}$}

Given a variety $X$ over $k$ and a subgroup $G\subset \aut_{K}(X)$, a $G$-structure is an algebraic structure $\xi$ on $X$ in the sense of \cite[\S 5]{giulio-angelo-moduli} such that $\aut_{K}(X,\xi)\subset \aut_{K}(X)$ is conjugate to $G$.

Given a finite subgroup $G\subset\PGL_{2}(K)$, if there exists a $G$-structure $\xi$ on $\PP^{1}$ which is not defined over its field of moduli then $G$ is cyclic of even order, see \cite[Theorem 5]{giulio-divisor} (while the result is stated for effective, reduced divisors, the proof works without modifications for any structure on $\PP^{1}$).

Similarly, if $G\subset\PGL_{3}(K)$ is finite, the existence of a $G$-structure on $\PP^{2}$ not defined over its field of moduli puts strong constraints on $G$. As we are going to prove, such a $G$-structure only exists if $G$ is critical.

\begin{theorem}\label{theorem:critical}
	Let $G\subset\PGL_{3}(K)$ be a finite subgroup and $\xi$ a $G$-structure over $\PP^{2}_{K}$ with field of moduli $k$. If $G$ is not critical, then $\xi$ descends to a structure over some Brauer-Severi surface over $k$. If $G$ is lucky, $\xi$ descends to $\PP^{2}_{k}$.
	
	On the other hand, if $G\subset\PGL_{3}(\bar{\QQ})$ is critical there exists a field $k$ of characteristic $0$ with algebraic closure $K$ and a $G$-structure $\xi$ on $\PP^{2}_{K}$ not defined over its field of moduli. If $G$ is not conjugate to $H_{3}$, we may choose $k$ so that it contains $\CC$.
\end{theorem}

The first part of Theorem~\ref{theorem:cycles} is a direct consequence of Theorem~\ref{theorem:critical}, while the second part follows from Theorem~\ref{theorem:critical} using \cite[Theorem 3]{giulio-fmod}.

We spend the rest of this section proving Theorem~\ref{theorem:critical}. We apply the two strategies described in \S\ref{sect:strategies}, or small variations of them, to each finite subgroup of $\PGL_{3}(K)$. 

While Theorem~\ref{theorem:critical} partially overlaps with a theorem of B. Huggins \cite[Theorem 6.4.8]{hugginsphd}, we do not use her result in our proof for two reasons. The first is that we work with arbitrary algebraic structures, whereas she only works with smooth plane curves. The second reason is that, in the large majority of cases, we prove that the curve (or structure) descends to a curve embedded in $\PP^{2}$ over the field of moduli, whereas she only proves that the abstract curve descends to the field of moduli. 

In the first half, we prove that if $G$ is lucky (resp. not critical) then $\xi$ descends to a structure on $\PP^{2}_{k}$ (resp. on some Brauer-Severi variety) by finding a rational point of $\bP_{\xi}$ whose corresponding singularity in $X/G$ is of type $\rR$ (resp. $\cG_{\xi}(k)\neq\emptyset$), this gives us a rational point of $\cP_{\xi}$ by the Lang--Nishimura theorem for tame stacks. 

In the second half, we construct the various counterexamples for $G$ critical.

\subsection{Type (A), not abelian}

Assume that $G\subset\GL_{2}(K)\subset\PGL_{3}(K)$ is of type {\bf (A)} and not abelian. Let $Z\subset G$ be the center and write $\cG\to\bar{\cG}$ for the rigidification of $\cG$ modulo the center of the inertia, see \cite[Appendix C]{dan-tom-angelo2008}. Essentially, we can pass to the quotient $G/Z$ at the level of gerbes; we have a natural identification $\bar{\cG}=\cB_{K}(G/Z)$.

Since $G\subset \GL_{2}(K)$, there is at least one line $L$ stabilized by $G$. If there is another one $L'$, then $p=L\cap L'$ is fixed and $G$ acts faithfully and diagonally on the tangent space of $p$, which is absurd since $G$ is not abelian. It follows that $L$ is the unique line stabilized by the whole $G$, it is a distinguished subspace and $L/G\subset\PP^{2}/G$ descends to a genus $0$ curve $C\subset \bP_{\xi}$. If $C$ is birational to $\PP^{1}$, since $\bP_{\xi}$ is normal we may find a rational point $c\in C(k)$ which is regular in $\bP_{\xi}$. Assume by contradiction that $C$ is a non-trivial Brauer-Severi curve over $k$.

Since $\bP_{\xi}$ is normal, by \cite[Corollary 3.2]{giulio-angelo-valuative} there is a rational map $C\dashrightarrow\bar{\cG}$. Let $\Delta\subset G\subset\GL_{2}(K)$ be the subgroup of diagonal matrices, observe that $\Delta\subset Z$ and that by construction the base change of $C\dashrightarrow\bar{\cG}$ to $K$ is the composition $L/(G/\Delta)\dashrightarrow \cB_{K}(G/\Delta)\to \cB_{K}(G/Z)$. In particular, the geometric fibers of $C\dashrightarrow\bar{\cG}$ are birational to $L/(Z/\Delta)$ and hence of genus $0$. This allows us to apply \cite[Proposition 2]{giulio-divisor}, which implies that $\bar{G}$ is cyclic. The fact that the quotient $\bar{G}$ of $G$ by its center is cyclic implies that $G$ is abelian, which is absurd. 

\subsection{Abelian, not conjugate to $H_{1}$}

If $G$ is abelian but not conjugate to $H_{1}$, it is diagonal by Lemma~\ref{corollary:abelian}, hence it is isomorphic to $C_{a}\times C_{an}$ for some positive integers $a,n$ and, up to conjugation, it is generated by three diagonal matrices of the form $\diag(\zeta_{a},1,1)$, $\diag(1,\zeta_{a},1)$, $\diag(\zeta_{an}^{b},\zeta_{an}^{d},1)$ for some integers $0\le b,d< an$ such that $\gcd(an,b,d)=1$.

If $a=1$, $b=d$ and $n$ is even, then $G$ is critical. If $a=1$, $b=d$ and $n\ge 3$ is odd, then the point $(0:0:1)$ is distinguished, it descends to a rational point of $\bP_{\xi}$ and the corresponding singularity is of type $\rR$ by \cite[Theorem 4]{giulio-tqs2}. If $a=n=1$, then $G=\{\id\}$ is not critical, and clearly $\cG_{\xi}=\spec k$ has a rational point.

Otherwise, the fixed locus of $G$ consists of the three points $(1:0:0)$, $(0:1:0)$ and $(0:0:1)$. Denote by $\rho_{1},\rho_{2},\rho_{3}:G\to\GL_{2}(K)$ the three corresponding faithful representations on the tangent spaces, each one splits uniquely as a sum of two characters. We say that two representations $\rho,\rho'$ of $G$ are equivalent if there exists an automorphism $\phi$ of $G$ such that $\rho$ and $\rho\circ\phi$ are isomorphic representations. There are three cases: the three representations $\rho_{1},\rho_{2},\rho_{3}$ are equivalent, only two of them are equivalent or they are pairwise non-equivalent.

\subsubsection{Three equivalent representations}\label{sect:3eq}

If the three representations are equivalent, then the subgroup of $C_{n}^{2}$ generated by $(b,d)$ is equal to the one generated by $(-d,b-d)$. In particular, $\gcd(b,n)=\gcd(d,n)=1$ since $\gcd(b,d,n)=1$. Up to multiplying by the inverse of $b$ modulo $n$, we may thus assume $b\cong 1\pmod{n}$. Furthermore, since $\diag(\zeta_{a},1,1)\in G$, we may reduce to the case $b=1$.

The subgroups of $C_{n}^{2}$ generated by $(1,d)$ and $(-d,1-d)$ are equal if and only if the determinant $d^{2}-d+1$ is congruent to $0$ modulo $n$. It follows that $G$ is not lucky, and it is critical if and only if $3\mid an$. If $3\mid an$, in the second half of the proof we will construct an example in which $\xi$ does not descend to any Brauer-Severi surface over $k$. Right now we are interested in the positive result, i.e. if $3\nmid an$ then $\xi$ descends to some Brauer-Severi surface, or equivalently $\cG_{\xi}(k)\neq\emptyset$.

Since $d^{2}-d+1\cong 1\pmod{n}$, then $n$ is odd (there is no such $d$ modulo $2$). The singularity of $\PP^{2}/G$ in the images of each of the three fixed points is cyclic of type $\frac1n(1,d)$, and it is of type $\rR$ since $n$ is odd \cite[Theorem 4]{giulio-tqs2}. Let $F\subset\PP^{2}$ be the fixed locus of $G$, we have that $F/G$ descends to a reduced, effective $0$-cycle $Z$ of degree $3$ on $\bP_{\xi}$, and the singularities in the geometric points of $Z$ are of type $\rR$.

If $\cG_{\xi}(k)=\emptyset$, since the singularities of $Z$ are of type $\rR$ then $Z(k)=\emptyset$ by the Lang--Nishimura theorem for tame stacks, hence $Z$ has only one point with residue field $k'/k$ of degree $3$. It follows that $\cG_{\xi}(k')\neq\emptyset$. Let $A$ be the band of $\cG_{\xi}$, it is a finite, $an$-torsion étale abelian group scheme over $k$. The non-neutral gerbe $\cG_{\xi}$ corresponds to a non-zero cohomology class $\psi\in\H^{2}(k,A)$ satisfying $3\psi=\operatorname{cor}_{k'/k}(\psi_{k'})=\operatorname{cor}_{k'/k}(0)=0\in\H^{2}(k,A)$. Since $A$ is $an$-torsion and $3\nmid an$, this implies $\psi=0$, which is absurd.

\subsubsection{Two equivalent representations}\label{sect:2eq}

If only two of the representations are equivalent, up to conjugation the ones corresponding to points $(1:0:0)$ and $(0:1:0)$, then we get the equality $b^{2}\cong d^{2}\pmod{n}$. In particular, $\gcd(b,n)=\gcd(d,n)=1$, and up to multiplying by the inverse of $b$ we may assume $b\cong1\pmod{n}$. Furthermore, since $\diag(\zeta_{a},1,1)\in G$, we may assume $b=1$.

The third fixed point $(0:0:1)$ is distinguished and it descends to a rational point of $\bP_{\xi}$. If $G$ is not critical, then by \cite[Theorem 4]{giulio-tqs2} the singularity in $(0:0:1)$ of $\PP^{2}/G$ is of type $\rR$, hence the rational point is liftable.

\subsubsection{Pairwise non-equivalent representations}

If the three representations are pairwise non-equivalent, then each of the fixed points is distinguished, and they descend to three rational points of $\bP_{\xi}$. If by contradiction $\cP_{\xi}(k)=\emptyset$, these three rational points are not liftable, hence the singularities of $\PP^{2}/G$ in $(0:0:1),(0:1:0),(1:0:0)$ are not of type $\rR$. The three singularities are of type $\frac1{n_{1}}(b,d)$, $\frac1{n_{2}}(b-d,-d)$ and $\frac1{n_{3}}(d-b,-b)$ respectively for some $n_{i}|n$, $\gcd(n_{1},b)=\gcd(n_{1},d)=1$ and similarly for $n_{2},n_{3}$. Let $n'=\gcd(n_{1},n_{2},n_{3})$ be the greatest common divisor, then $\gcd(n',b)=\gcd(n',d)=\gcd(n',b-d)=1$. Since the three rational points are not liftable, $n'$ is even by \cite[Theorem 4]{giulio-tqs2}. This is absurd, since it implies that $b$, $d$ and $b-d$, being coprime with $n'$, are all odd.

\subsection{\boldsymbol{$H_{1}$}}\label{sec:H1}

Assume that $G$ is conjugate to $H_{1}\simeq C_{3}\times C_{3}$: it is not critical, nor lucky. We want to show that $\cG_{\xi}(k)\neq\emptyset$. If $g\in G$ is non-trivial, its fixed locus has exactly three points, and if $h\in G\setminus \left<g\right>$ then $g$ permutes the three fixed points of $h$. It follows that the union $F$ of the fixed loci of the non-trivial elements of $G$ is a distinguished subset with $12$ points, and $F/G\subset \PP^{2}/G$ descends to a reduced $0$-cycle $Z\subset X$ of degree $4$. In particular, there exists a finite extension $k'/k$ of degree prime with $3$ such that $Z(k')\neq \emptyset$.

The stabilizer of each point of $F$ is a cyclic group of order $3$, hence the four singularities of $\PP^{2}/G$ in the points of $F/G$ are all of type $\rR$ by \cite[Theorem 4]{giulio-tqs2}. Since $Z(k')\neq\emptyset$, by the Lang--Nishimura theorem for tame stacks we have that $\cG_{\xi}(k')\neq\emptyset$. Now observe that $\cG_{\xi}$ is abelian, and thus it corresponds to a cohomology class $\psi\in\H^{2}(k,A)$ where $A$, the band of $\cG_{\xi}$, is a $3$-torsion finite group scheme. We have that $[k':k]\psi=\operatorname{cor}_{k'/k}(\psi_{k'})=0$, hence $\psi=0$ since $[k':k]$ is prime with $3$ and $\H^{2}(k,A)$ is $3$-torsion. It follows that $\cG_{\xi}$ is neutral.

\subsection{Type (B)}

Up to conjugation, we may assume that $G=D\rtimes C_{3}$, where $D$ is non-trivial and diagonal, and $C_{3}$ is generated by a permutation matrix of order $3$. The fact that $C_{3}$ normalizes $D$ implies that $D$ has the form studied in \S\ref{sect:3eq}, and as such it is generated by three matrices $\diag(\zeta_{a},1,1)$, $\diag(1,\zeta_{a},1)$, $\diag(\zeta_{an},\zeta_{an}^{d},1)$ with $d^{2}-d+1\cong 0\pmod{n}$. 

Thanks to the preceding case, we may assume that $G$ is not $H_{1}$: under this assumption, $D\subset G$ is the only diagonalizable subgroup of index $3$, and the fixed locus $F$ of $D$ is a distinguished subset. Since $D$ is non-trivial, either $a\neq 1$ or $d$ is not congruent to $0,1$ modulo $n$: in both cases $F$ has exactly three points, $G$ acts transitively on it and $F/G\subset \PP^{2}/G$ descends to a rational point $p\in \bP_{\xi}(k)$. The singularity of $\PP^{2}/G$ in the point $F/G$ is cyclic of type $\frac 1n(1,d)$. 

If $d^{2}\not\cong 1\pmod{n}$, then the singularity is of type $\rR$ \cite[Theorem 4]{giulio-tqs2}. If $d^{2}\cong 1\pmod{n}$, since $d^{2}-d+1\cong 0\pmod{n}$, then $d\cong 2\pmod{n}$ and $n$ is either $1$ or $3$. In both cases the singularity is again of type $\rR$ by \cite[Theorem 4]{giulio-tqs2}.

\subsection{Type (C), not conjugate to $H_{2}$}

Assume that $G$ is of type {\bf (C)} and not conjugate to $H_{2}$. Up to conjugation there is a diagonal, normal subgroup $D\subset G$ with $G/D\simeq S_{3}$ and $M_{1}\in G$. 

Let us show that $D$ is the only normal, diagonal subgroup of index $6$. If $D'$ is another such subgroup, then the image of $D'$ in $S_{3}$ is normal, abelian and non-trivial, i.e. it is $C_{3}\subset S_{3}$. It follows that $D'$ contains an element of the form $M_{1}N$, where $N\in D$ is diagonal. If $D'\cap D$ is non-trivial, since $M_{1}N\in D'$ then $D'$ has no fixed points, hence it is not diagonal. If $D'\cap D$ is trivial, then $|D'|=|D|=3$ and $D$ is generated by $M_{0}$, since $\left<M_{0}\right>$ is the only diagonal subgroup of order $3$ normalized by $M_{1}$. It follows that $G$ is conjugate to $H_{2}$.

This implies that the fixed locus $F$ of $D$ is distinguished, it contains $3$ points and $F/G$ descends to a rational point $p\in \bP_{\xi}(k)$. We want to show that $p$ is liftable. Let $x\in F$ be a point with stabilizer $G_{x}$, we have that $G_{x}$ is an extension of $C_{2}$ by $D$.

As in the previous case, $D\simeq C_{a}\times C_{an}$ with generators $\diag(\zeta_{a},1,1)$, $\diag(1,\zeta_{a},1)$, $\diag(\zeta_{an},\zeta_{an}^{d},1)$, and again we have $d^{2}-d+1\cong 0\pmod{n}$. Now we have another condition, which is the fact that $D$ is normalized by a matrix $M=\begin{psmallmatrix}  & \beta &  \\ \alpha &  &  \\  &  & 1 \end{psmallmatrix}$, and this implies $d^{2}\cong 1\pmod{n}$, hence $n$ is either $1$ or $3$.

If $n=1$, then $M^{2}\in C_{a}\times C_{a}$ and hence $\alpha\beta$ is a power of $\zeta_{a}$. Up to multiplying $M$ by a suitable element of $C_{a}\times C_{a}$, we may thus assume that $\alpha\beta=1$. If $\alpha\beta=1$ then $M$ is a pseudoreflexion, and since $C_{a}\times C_{a}$ is generated by pseudoreflexions we get that $G_{x}$ is generated by pseudoreflexions as well. It follows that $p$ is smooth, and in particular liftable.

If $n=3$, then $G_{x}$ is an extension of $D_{3}$ by $C_{a}\times C_{a}$, hence it is $\rR_{2}$ by \cite[Propositions 6.17, 6.20]{giulio-angelo-moduli} and $p$ is liftable.

\subsection{$\boldsymbol{H_{3}}$, $\boldsymbol{\zeta_{12}}\in k$}

Observe that $\QQ(\zeta_{12})=\QQ(\zeta_{3},\zeta_{4})=\QQ(\sqrt{3},i)$, hence $\sqrt{3}\in k$. We have $\det M_{0}=\det M_{1}=1$, $\det M_{2}=-1$, $\det M_{3}=\det M_{4}=3\sqrt{3}$, hence $\det M_{i}\in k^{*3}$ for every $i\le 4$. Since $M_{i}\in\GL_{3}(k)$ for every $i$, we may thus find $a_{i}\in k^{*}$ such that $M_{i}'=a_{i}M_{i}\in \SL_{3}(k)$ for every $i\le 4$. Let $H_{i}'$ be the inverse image of $H_{i}$ in $\SL_{3}$, it follows that $H_{i}'$ and $H_{i}$ are constant group schemes over $k$ for $i\le 4$. Furthermore, $M_{4}'^{4}=\id\in\SL_{3}(k)$.

Since $H_{5}/H_{2}\simeq A_{4}$ and $H_{4}/H_{2}\simeq \operatorname{Kl}$, Lemma \ref{lemma:H1norm} implies that the normalizer of $H_{3}$ in $\PGL_{3}(K)$ is $H_{4}$. Observe that $\H^{1}(k,H_{4}')\to\H^{1}(k,H_{4}/H_{3})$ is surjective, because there is an homomorphism $C_{4}\to H_{4}'$ defined by $1\mapsto M_{4}'$ which lifts the projection $H_{4}'\to H_{4}/H_{3}\simeq C_{2}$, and $\H^{1}(k,C_{4})\to\H^{1}(k,C_{2})$ is surjective since $\zeta_{4}\in k$. 

Observe that the image of $\H^{1}(k,H_{4}')\to\H^{1}(k,H_{4})$ is contained, by construction, in the kernel of $\H^{1}(k,H_{4})\to\H^{1}(k,\PGL_{3})$, since the composition $\H^{1}(k,\SL_{3})\to\H^{1}(k,\GL_{3})\to\H^{1}(k,\PGL_{3})$ is trivial (notice that these are sets, not groups, but they have a preferred object and it makes sense to consider kernels). We may then apply \cite[Theorem 16]{giulio-fmod} and obtain that every $H_{3}$-structure with field of moduli $k$ descends to $\PP^{2}_{k}$.

\subsection{$\boldsymbol{H_{4}}$}

Assume $G=H_{4}$, we have that $H_{1}\simeq C_{3}\times C_{3}$ is the only $3$-Sylow subgroup (it is normal) and hence it is characteristic. It follows that the union $F$ of the fixed loci of the non-trivial elements of $H_{1}$ is a finite, distinguished subset of degree $12$ which is the union of $4$ triangles corresponding to the $4$ cyclic subgroups of $C_{3}\times C_{3}$ (see \ref{sec:H1}). The action of $H_{4}$ on $F$ is transitive since $H_{1}$ acts transitively on each triangle and $H_{4}/H_{2}\subset H_{5}/H_{2}\simeq A_{4}$ acts as the Klein group, and hence transitively, on the set of $4$ triangles.

It follows that $F/H_{4}$ descends to a rational point $p\in \bP_{\xi}(k)$. The stabilizer of a point of $F$ has order $72/12=6$ and is isomorphic $D_{3}$ (it's easy to see that the element of order $2$ maps to $-\operatorname{Id}\in\SL(2,3)$). Since $D_{3}$ is $\rR_{2}$ \cite[Proposition 6.17]{giulio-angelo-moduli}, then $p$ is liftable.

\subsection{$\boldsymbol{H_{5}}$}

Assume $G=H_{5}$, it has order $|H_{5}|=216=2^{3}\cdot 3^{3}$. Let $g\in H_{5}$ be an element mapping to $-\operatorname{Id}\in\SL(2,3)$, conjugation by $g$ acts as $-1$ on $H_{1}$ and as the identity on $H_{5}^{\rm ab}$. Because of this, and since $H_{1}$ has odd order, the homomorphism $H_{1}\to H_{5}^{\rm ab}$ is trivial (an element in the image has odd order and is equal to its inverse, i.e. it is the identity). Moreover, we have a projection $H_{5}^{\rm ab}\to H_{5}/H_{4}\simeq C_{3}$. This implies that $9$ is the largest power of $3$ dividing the order of $[H_{5},H_{5}]$, hence $H_{1}$ is the unique $3$-Sylow subgroup of $[H_{5},H_{5}]$ (it is normal), which in turn implies that $H_{1}$ is characteristic in $H_{5}$.

Since $H_{1}$ is characteristic, the union $F$ of the fixed loci of the non-trivial elements of $H_{1}$ is a finite, distinguished subset of degree $12$. As in the previous case, the action of $H_{5}$ on $F$ is transitive and hence $F/H_{5}\subset \PP^{2}/H_{5}$ descends to a rational point $p\in \bP_{\xi}(k)$. Let $x\in F$ be a point and $G_{x}\subset H_{5}$ its stabilizer, the degree of $G_{x}$ is $216/12=18$. Since $H_{4}\subset H_{5}$ is normal, $H_{4}\cap G_{x}$ is a normal subgroup of $G_{x}$, it is isomorphic to $D_{3}$ by the argument given in the previous case and $G_{x}/(H_{4}\cap G_{3})\simeq C_{3}$. By \cite[6.17, 6.19, 6.20]{giulio-angelo-moduli} we get that $G_{x}$ is of type $\rR_{2}$, hence $p$ is liftable.

\subsection{$\boldsymbol{A_{5}}$ or $\boldsymbol{A_{6}}$}

If $G\simeq A_{n}$ for $n=5,6$, let us first show that the eigenvalues of $(1,2,3)\in A_{n}$ as an element of $G\subset \PGL_{3}(K)$ are pairwise different. If they are not, up to conjugation we may assume that the matrix of $(1,2,3)$ is $\diag(\zeta_{3},\zeta_{3},1)$. Since $(1,2)(4,5)$ normalizes $\left<(1,2,3)\right>$, then it must stabilize its fixed locus, namely $(0:0:1)$ and $\{(s:t:0)\}$. It follows that $(1,2)(4,5)\in\GL_{2}\subset\PGL_{3}$ and hence it commutes with $(1,2,3)=\diag(\zeta_{3},\zeta_{3},1)$, which is absurd.

Every automorphism of $A_{n}$ maps a $3$-cycle to a $3$-cycle. For $n=5$, this is obvious. For $n=6$, the only other elements of order $3$ are double $3$-cycles, but there is only one conjugacy class of $3$-cycles with $40$ elements and two conjugacy classes of double cocycles with $20$ elements each. Furthermore, the conjugacy action of $A_{n}$ on $3$-cycles is transitive for both $n=5,6$.

Let $x_{1},x_{2},x_{3}$ be the fixed points of $(1,2,3)$. Since $(1,2)(4,5)$ normalizes $\left<(1,2,3)\right>$ but it does not commute with $(1,2,3)$, it fixes exactly one of the $x_{i}$s, say $x_{1}$, and swaps $x_{2},x_{3}$. Let $F$ be the orbit of $x_{1}$ and $F'$ the orbit of $x_{2},x_{3}$: since there is only one conjugacy class of $3$-cycles, the union of the fixed loci of the $3$-cycles is $F\cup F'$ and we have either $F=F'$ or $|F'|=2|F|\neq |F|$. In any case, $F$ is a distinguished subset, $x_{1}\in F$ and $A_{n}$ acts transitively on it. It follows that $F/A_{n}$ descends to a rational point $p\in \bP_{\xi}(k)$. 

Let $G_{x_{1}}$ be the stabilizer of $x_{1}$, since it contains both $(1,2,3)$ and $(1,2)(4,5)$ then it has trivial center (their respective centralizers have trivial intersection). The action of $G_{x_{1}}$ on the tangent space of $x_{1}$ gives us an embedding $G_{x_{1}}\subset\GL_{2}(K)$. Since $G_{x_{1}}$ has trivial center, the composition $G_{x_{1}}\to\GL_{2}(K)\to\PGL_{2}(K)$ is injective, and hence $G_{x_{1}}$ is isomorphic either to $D_{n},A_{4},S_{4}$ or $A_{5}$. This implies that $G_{x_{1}}\simeq D_{n}$ with $n$ odd, since in all the other cases $G_{x_{1}}\subset\GL_{2}(K)$ would contain a subgroup isomorphic to $C_{2}\times C_{2}$ and hence the matrix $-\operatorname{Id}\in\GL_{2}(K)$, which is central. It follows that $G_{x_{1}}$ is $\rR_{2}$ by \cite[Propositions 6.17]{giulio-angelo-moduli}, hence $x$ is liftable.

\subsection{$\boldsymbol{\PSL(2,7)}$}

Assume $G\simeq\PSL(2,7)$, we have $|\PSL(2,7)|=168=2^{3}\cdot 3\cdot 7$. We may choose an embedding of $\PSL(2,7)$ in $\PGL_{3}(K)$ so that it contains the matrix $M_{7}=\diag(1,\zeta_{7},\zeta_{7}^{3})$ and the permutation matrix $M_{1}$ of order $3$ given in \S\ref{sect:hessian}, see \cite[Chapter XII, \S 123]{mbd}. Furthermore, we have $M_{1}^{-1}M_{7}M_{1}=M_{7}^{4}$.

Let $n_{7}$ be the number of $7$-Sylow subgroups, by Sylow's third theorem $n_{7}$ divides $2^{3}\cdot 3$ and $n_{7}\cong 1\pmod{7}$, i.e. $n_{7}$ is either $1$ or $8$. The upper or lower triangular matrices in $\PSL(2,7)$ with $1$-s on the diagonal are two different $7$-Sylow subgroups, hences $n_{7}=8$. It follows that the normalizer of a $7$-Sylow subgroup has order $168/8=21$. Since $M_{1}$ normalizes $M_{7}$ and has order $3$, then the normalizer of $\left<M_{7}\right>$ is $\left<M_{7},M_{1}\right>$ and hence the centralizer of $M_{7}$ is $\left<M_{7}\right>$.

Let $F$ be the union of the fixed loci of all the $7$-Sylow subgroups, it is a distinguished subset. Since $M_{1}$ permutes the three fixed points of $M_{7}$ and $\PSL(2,7)$ acts transitively by conjugation on the $7$-Sylow subgroups, then the action of $G$ on $F$ is transitive, and $F/G$ descends to a rational point $p\in\bP_{\xi}(k)$.

Let $x=(0:0:1)\in F$ and let $G_{x}$ be its stabilizer, then $M_{7}\in G_{x}$. Since the centralizer of $M_{7}$ is $\left<M_{7}\right>$, either $G_{x}\simeq C_{7}$ or $G_{x}$ has trivial center. If $G_{x}\simeq C_{7}$, then it is $\rR_{2}$ by \cite[Theorem 6.19]{giulio-angelo-moduli} and hence $p$ is liftable. If $G_{x}$ has trivial center, the homomorphism $G_{x}\to\GL_{2}(K)\to\PGL_{2}(K)$ given by the action on the tangent space of $p$ is injective, hence $G_{x}$ is dihedral since the other finite subgroups of $\PGL_{2}(K)$ are either abelian or of order prime with $7$. If $G_{x}$ is dihedral, then it is of type $\rR_{2}$ by \cite[Propositions 6.17]{giulio-angelo-moduli}, hence we get that $p$ is liftable in this case, too.

\vspace{1em}

This concludes the proof of the first half of the theorem. Let us now construct the counterexamples. As we explained in \S\ref{sect:strategies}, we will do this by applying the first half of \cite[Theorem 4]{giulio-fmod}.

\subsection{$\boldsymbol{C_{a}\times C_{an}}$, $\boldsymbol{d^{2}-d+1\cong 0\pmod{n}}$, $\boldsymbol{3|an}$}

With an abuse of notation, for every integer $m$ we write $\Delta$ for the diagonal subgroup of $C_{m}^{3}$, we have a preferred embedding $C_{m}^{3}/\Delta\subset\PGL_{3}(K)$ by diagonal matrices.

\subsubsection{$\boldsymbol{3|n}$}

Consider the action by permutation of $C_{3}$ on $C_{3}^{3}/\Delta$, the semidirect product $E=(C_{3}^{3}/\Delta)\rtimes C_{3}$ is a non-abelian group of order $27$ with a central subgroup $A=\left<(1,2,0)\right>\subset C_{3}^{3}/\Delta\subset E$ such that $E/A\simeq C_{3}^{2}$. Observe that $E$ is $3$-torsion: if $(c,\phi)$ is an element, then $3(c,\phi)=(c+\phi(c)+\phi^{2}(c),0)$ is trivial since clearly $c+\phi(c)+\phi^{2}(c)\in \Delta\subset C_{3}^{3}$. Since $E/A\simeq C_{3}^{2}$ is abelian, $A$ is central and $E$ is non-abelian, the extension is non-split.

Since $3|n$, then $d^{2}-d+1\cong0\pmod{n}$ implies that $d\cong 2\pmod{3}$ and $9\nmid n$. Let $N_{0}$ be the subgroup of $C_{an}^{3}/\Delta\subset\PGL_{3}(K)$ generated by $G$ and by $C_{3a}^{3}/\Delta\subset\PGL_{3}(K)$, clearly $G\subset N_{0}\subset C_{an}^{3}/\Delta$ and $N_{0}/G\simeq C_{3}$. The natural projection $C_{an}^{3}/\Delta\to C_{3}^{3}/\Delta$ maps $G$ onto $A=\left<(1,2,0)\right>$ since $d\cong 2\pmod{3}$, while $N_{0}\to C_{3}^{3}/\Delta$ is surjective by construction.

The condition $d^{2}-d+1\cong 0\pmod{n}$ implies that the action by permutation of $C_{3}$ on $C_{an}^{3}/\Delta$ restricts to $G$, and hence to $N_{0}$. Denote by $N$ the semidirect product $N_{0}\rtimes C_{3}\subset\PGL_{3}(K)$ generated by $N_{0}$ and by $M_{1}$, clearly $N/G\simeq C_{3}^{2}$ and we have a commutative diagram of short exact sequences
\[\begin{tikzcd}
	1\rar	&	G\rar\dar	&	N\rar\dar	&	N/G\rar\ar[d,"\sim"]		&	1	\\
	1\rar	&	A\rar		&	E\rar		&	C_{3}^{2}\rar	&	1
\end{tikzcd}\]

Choose $k=\CC((s))((t))$, and consider $G\subset N\subset\PGL_{3}(k)$ as constant group schemes, so that torsors correspond to homomorphisms from the Galois group. The natural projection $\gal(K/k)=\hat{\ZZ}^{2}\to C_{3}^{2}$ defines an $N/G$-torsor $T$ which does not lift to $E$, since $E$ is $3$-torsion and hence a lifting $\hat{\ZZ}^{2}\to E$ would give us a section $C_{3}^{2}\to E$. In particular, $T$ does not lift to $N$, as desired.

\subsubsection{$\boldsymbol{3\nmid n}$, $\boldsymbol{3|a}$}

Now assume $3\nmid n$, $3|a$. Since $3$ does not divide $n$ and $2^{2}-2+1\cong0\pmod{3}$, we may assume that $d$ is such that $d^{2}-d+1\cong 0\pmod{3n}$. Let $N_{0}\subset C_{3an}^{3}/\Delta\subset\PGL_{3}(K)$ be the subgroup generated by $(3n,0,0)$, $(0,3n,0)$ and $(1,d,0)$, then $G\subset N_{0}$ and $N_{0}/G\simeq C_{3}$. Since $d^{2}-d+1\cong 0\pmod{3n}$, then both $G$ and $N_{0}$ are stabilized by the action of $C_{3}$ on $C_{3an}^{3}/\Delta$. Let $N$ be the semidirect product $N_{0}\rtimes C_{3}$, again we have $N/G\simeq C_{3}^{2}$.

Observe that every element of $N_{0}\subset C_{3an}^{3}/\Delta$ fixed by $C_{3}$ is contained in $C_{3}^{3}/\Delta\subset C_{3an}^{3}/\Delta$, and since $3|a$ then $C_{3}^{3}/\Delta\subset G$. Because of this, there are no abelian subgroups of $N$ which map surjectively on $N/G\simeq C_{3}^{2}$. Hence, if we choose $k=\CC((s))((t))$, $\gal(K/k)=\hat{\ZZ}^{2}\to N/G=C_{3}^{2}$ as in the previous case, there is no lifting $\gal(K/k)\to N$.

\subsection{$\boldsymbol{C_{a}\times C_{a2^{b}n}}$, $\boldsymbol{d^{2}\cong 1\pmod{n}}$, $\boldsymbol{d\cong\pm1\pmod{2^{b}}}$}

Assume that $G\simeq C_{a}\times C_{a2^{b}n}$ is generated by $\diag(\zeta_{a},1,1)$, $\diag(1,\zeta_{a},1)$ and $\diag(\zeta_{a2^{b}n},\zeta_{a2^{b}n}^{d},1)$ for some positive integers $a,b,n,d$ with $d^{2}\cong 1\pmod{n}$, $d\cong\pm1\pmod{2^{b}}$ and $n$ odd.

Consider the semidirect product $C_{2^{b}}^{2}\rtimes C_{2}$ where the action swaps the coordinates, define $E_{1}$ as the subgroup generated by $(1,1,0),(2^{b-1},0,0),(0,0,1)$ and $E_{-1}$ as the one generated by $(1,-1,0),(2^{b-1},0,0),(0,0,1)$, and let $A_{\pm1}\subset E_{\pm 1}$ be the subgroup generated by $(1,\pm 1,0)$. We have that $E_{\pm 1}$ is an extension of $C_{2}^{2}$ by $A_{\pm 1}\simeq C_{2^{b}}$, and there is no abelian subgroup of $E_{\pm1}$ mapping surjectively on $C_{2}^{2}=E_{\pm 1}/A_{\pm1}$.

Let $N_{0}\subset\PGL_{3}(K)$ be the subgroup generated by $G$ and by $\diag(\zeta_{a2^{b}n}^{2^{b-1}n},1,1)=\diag(\zeta_{2a},1,1)$, it is abelian and $G$ has index $2$ in $N_{0}$. Observe that $\diag(\zeta_{2a},\zeta_{2a}^{d},1)\in G$, and since $d$ is odd then $\diag(\zeta_{2a},\zeta_{2a}^{-1},1)\in G$ and $\diag(1,\zeta_{2a},1)\in N_{0}$. Since $d^{2}\cong 1\pmod{2^{b}n}$, a permutation matrix swapping the first two coordinates normalizes both $G$ and $N_{0}$, let $N\simeq N_{0}\rtimes C_{2}$ be the subgroup generated by this permutation matrix and $N_{0}$. 
 
Consider the natural projection $N_{0}\to C_{2^{b}}^{2}$, if $d\cong 1\pmod{2^{b}}$ then it extends to a surjective map $N\to E_{1}$, while if $d\cong-1\pmod{2^{b}}$ it extends to a surjective map $N\to E_{-1}$. In both cases, $G$ maps surjectively on $A_{\pm 1}$. We thus have a commutative diagram of short exact sequences
\[\begin{tikzcd}
	1\rar	&	G\rar\dar			&	N\rar\dar			&	N/G\rar\ar[d,"\sim"]		&	1	\\
	1\rar	&	A_{\pm 1}\rar		&	E_{\pm1}\rar		&	C_{2}^{2}\rar	&	1
\end{tikzcd}\]

We may then choose $k=\CC((s))((t))$, $\gal(K/k)=\hat{\ZZ}^{2}\to N/G\simeq C_{2}^{2}$ similarly to the previous cases, there is no lifting $\gal(K/k)\to N$.

\subsection{\boldsymbol{$H_{2}$}}

Again, we choose $k=\CC((s))((t))$ and use constant group schemes. Observe that $H_{2}\subset H_{4}$ is normal, $H_{4}/H_{1}\subset\SL(2,3)$ is the quaternion group, $H_{2}/H_{1}\subset H_{4}/H_{1}$ is the center and $H_{4}/H_{2}\simeq C_{2}^{2}$. Since the quaternion group has no abelian subgroups of rank $2$, the natural projection $\gal(K/k)=\hat{\ZZ}^{2}\to H_{4}/H_{2}=C_{2}^{2}$ does not lift to $H_{4}$.

\subsection{\boldsymbol{$H_{3}$}, \boldsymbol{$\zeta_{12}\notin k$}}

Take $k=\RR$, and consider the group scheme structures $\fH_{3},\fH_{4}$ on $H_{3},H_{4}$ given in \S\ref{sect:hessch}. We have a factorization 
\[\fH_{4}\to\fH_{4}/\fH_{2}\simeq\fK\to\fH_{4}/\fH_{3}\simeq C_{2}.\]
It is enough to show that $\H^{1}(\RR,\fK)$ is trivial, since this implies that the only non-trivial $C_{2}$-torsor over $\RR$ does not lift to $\fH_{4}$.

By contradiction, let $T\to\spec \RR$ be a non-trivial $\fK$-torsor. Since it is non-trivial and has degree $4$, then clearly $T=\spec(\CC\times \CC)$ as a scheme and $\fK\subset\fS_{4}$ where $\fS_{4}$ is the group scheme of automorphisms of the étale scheme $T$; the group scheme $\fS_{4}$ is a twisted for of $S_{4}$. The two automorphisms $(a,b)\mapsto(b,a)$ and $(a,b)\mapsto (\bar{a},\bar{b})$ of $\spec(\CC\times \CC)$ act on the four geometric points of $T$ as different double transpositions, hence we have an embedding $\operatorname{Kl}\subset\fS_{4}$, where $\operatorname{Kl}$ is the Klein group with trivial Galois action. It follows that $\fK=\operatorname{Kl}$, which is absurd: since $\zeta_{3}\not\in \RR$, by construction the Galois action on $\fK(\CC)$ is non-trivial.

\section{Plane curves}

Let $j:C\hookrightarrow \PP^{2}_{K}$ be a smooth plane curve of degree $d\ge 4$ defined over $K$. The embedding in $\PP^{2}_{K}$ is unique up to composition with elements of $\PGL_{3}(K)$ \cite[Appendix A, \S 1, Exercise 18]{acgh1} and hence $\aut(C)=\aut(\PP^{2}_{K},C)$. The field of moduli $k_{(\PP^{2},C)}$ of the pair clearly contains the field of moduli $k_{C}$ of $C$. Let $\cG_{(\PP^{2},C)}\to\spec k_{(\PP^{2},C)}$ be the residual gerbe of $(\PP^{2},C)$ and $\cG_{C}\to\spec k_{C}$ the residual gerbe of $C$, there is a natural forgetful morphism $\cG_{(\PP^{2},C)}\to\cG_{C}$.

\begin{lemma}\label{lemma:plane}
	The fields of moduli $k_{(\PP^{2},C)}$ and $k_{C}$ are equal, and $\cG_{(\PP^{2},C)}\to\cG_{C}$ is an isomorphism.
\end{lemma}

\begin{proof}
	Let $\sigma\in\gal(K/k_{C})$ be an element, there exists an isomorphism $\phi:C\to\sigma^{*}C$. Choose $\tau:\PP^{2}_{K}\to\sigma^{*}\PP^{2}_{K}$ any isomorphism. The composition
	\[\tau^{-1}\circ \sigma^{*}j\circ\phi:C\to\PP^{2}\]
	is a plane embedding of $C$, hence it is equal to $g\circ j$ for some $g\in\PGL_{3}(K)$. We thus have a commutative diagram
	\[\begin{tikzcd}
		C\rar[hook, "j"]\dar["\phi"]			&	\PP^{2}\dar["\tau\circ g"]	\\
		\sigma^{*}C\rar[hook,"\sigma^{*}j"]		&	\sigma^{*}\PP^{2}
	\end{tikzcd}\]
	which shows that $\sigma\in\gal(K/k_{(\PP^{2},C)})$ and hence $k_{C}=k_{(\PP^{2},C)}$. The fact that $\cG_{(\PP^{2},C)}\to\cG_{C}$ is an isomorphism can now be checked after base changing from $k_{C}=k_{(\PP^{2},C)}$ to $K$, where it follows from the equality $\aut(C)=\aut(\PP^{2},C)$.
\end{proof}

Up to enlarging $k$, we may then assume $k_{(\PP^{2},C)}=k_{C}=k$. In particular, we get that $C$ is defined over $k$ if and only if the pair $(\PP^{2},C)$ is defined over $k$. As a direct consequence of Lemma~\ref{lemma:plane}, we obtain a new proof of the following result by J. Roè and X. Xarles.

\begin{corollary}[{\cite[Theorem 5]{roe-xarles}}]\label{corollary:brauer}
	For every model $\fC$ of $C$ over $k$, there exists a unique Brauer-Severi surface $P_{\fC}$ over $k$ with an embedding $\fC\hookrightarrow P_{\fC}$. 
\end{corollary}

\begin{remark}
	Notice that the stabilizer of a point of $C$ acts faithfully on the tangent space, in particular it is cyclic. Because of this, if $L\subset\PP^{2}_{K}$ is a line fixed by a non-trivial element $g$ of $\aut(C)$ then $C$ has normal crossing with $L$: if $p\in L\cap C$ and $L$ is tangent to $C$ in $p$, then $g$ acts trivially on the tangent space of $C$ in $p$, which is absurd. We are going to use these observations repeatedly without further justification.
\end{remark}

In Theorem~\ref{theorem:cycles} we have proved that the automorphism group of a cycle in $\PP^{2}$ not defined over the field of moduli must have a very specific form. In the next two propositions, we show that if the cycle is a smooth curve then we get even tighter constraints on the automorphism group.

\begin{proposition}\label{proposition:crit1}
	Let $k$ be a field of characteristic $0$ with algebraic closure $K$, $C\subset\PP^{2}_{K}$ a smooth plane curve over $K$ of degree $\ge 4$ with field of moduli $k$.
	
	Assume that $C$ is not defined by a homogeneous polynomial with coefficients in $k$, and that $\aut(C)$ has the form $C_{a}\times C_{an}$ generated by $\diag(\zeta_{a},1,1)$, $\diag(1,\zeta_{a},1)$, $\diag(\zeta_{an},\zeta_{an}^{e},1)$ for some positive integers $a,e,n$ with $e^{2}-e+1\cong 0\pmod{n}$. 
	
	Such an $e$ exists if and only if $n$ is odd and $-3$ is a square modulo $n$. Furthermore, one of the following holds.
	\begin{itemize}
		\item $an\mid d$, or
		\item $a=1$, $n\mid d^{2}-3d+3$ and $n\neq d^{2}-3d+3$.
	\end{itemize}
\end{proposition}

\begin{proof}
	Since $e^{2}-e+1\cong 0\pmod{n}$, we get that $n$ is odd: the equation has no solutions modulo $2$. It follows that the singularities in the three points $(0:0:1)$, $(0:1:0)$ and $(1:0:0)$ of $\PP^{2}/\aut(C)$ are of type $\rR$ thanks to \cite[Theorem 4]{giulio-tqs2}, since their local fundamental group has odd degree $n$. They form a distinguished subset, hence they define a $0$-cycle $Z$ of degree $3$ on the compression $\bP$.
	
	If $Z$ contains a rational point, then $\cP(k)\neq\emptyset$ by the Lang--Nishimura theorem for stacks \cite[Theorem 4.1]{giulio-angelo-valuative} applied to $\tilde{\bP}\dashrightarrow\cP$, where $\tilde{\bP}\to P$ is a resolution of singularities. This is absurd, since we are assuming that $C$ is not defined by a homogeneous polynomial with coefficients in $k$. Hence, $Z=\{z\}$ contains only one point $z$ with $[k(z):k]=3$. There are two cases: either $\bC\subset\bP$ contains $z$, or not. Equivalently, either $C$ contains the three points $(0:0:1)$, $(0:1:0)$, $(1:0:0)$, or none.
	
	{\bf Case 1: the three points are not in $C$.} Consider the line $L$ of points $(x:0:z)$, the group $\aut(C)=C_{a}\times C_{an}$ maps $L$ to itself. There are two fixed points on $L$, and all the other orbits have order $an$. By assumption, $C$ does not contain the fixed points. This implies that $an$ divides the degree of each orbit in $C\cap L$ considered as a $0$-cycle with multiplicities, this in turn implies that $an$ divides the degree of the $0$-cycle $C\cap L$, which is $d$.
	
	{\bf Case 2: the three points are in $C$.} Since the stabilizers of points of $C$ are cyclic and $\aut(C)$ fixes $(0:0:1)$, then $a=1$. Since $a=1$, the three points $(0:0:1)$, $(0:1:0)$ and $(1:0:0)$ are the only points of $\PP^{2}$ with non-trivial stabilizer. It follows that the ramification divisor of $C\to C/\aut(C)$ has degree $3(n-1)$. By Riemann--Hurwitz, we get
	\[d(d-3)=2n(h-1)+3(n-1),\]
	or equivalently
	\[d^{2}-3d+3=n(2h+1).\]
	It remains to show that $n\neq d^{2}-3d+3$, or equivalently $h\neq 0$. Assume by contradiction that $h=0$. The three points $(0:0:1)$, $(0:1:0)$ and $(1:0:0)$ are fixed and form a distinguished subset, hence they define a divisor of degree $3$ on the compression $\bC$ of $C$. Since $\bC$ is a twisted form of $C/\aut(C)$, it has genus $h=0$, hence $\bC$ is isomorphic to $\PP^{1}$ since it has a divisor of odd degree. We thus get a rational map $\PP^{1}\dashrightarrow\cG$, and hence $\cG(k)\neq \emptyset$ since $\PP^{1}(k)$ is dense, thus giving a contradiction.	
	
	To conclude, we have to show that a solution to $e^{2}-e+1\cong 0\pmod{n}$ exists if and only if $n$ is odd and $-3$ is a square modulo $n$. The fact that $n$ must be odd is obvious, since there are no solutions modulo $2$. Assuming that $n$ is odd, we can multiply by $4$ and obtain the equation
	\[(2e-1)^{2}\cong -3\pmod{n}\]
	which clearly has solutions if and only if $-3$ is a square modulo $n$.
\end{proof}

\begin{proposition}\label{proposition:crit2}
	Let $k$ be a field of characteristic $0$ with algebraic closure $K$, $C\subset\PP^{2}_{K}$ a smooth plane curve over $K$ of degree $\ge 4$ with field of moduli $k$. 
	
	Assume that $C$ is not defined by a homogeneous polynomial with coefficients in $k$, and that $\aut(C)$ has the form $C_{a}\times C_{2^{b}an}$ generated by $\diag(\zeta_{a},1,1)$, $\diag(1,\zeta_{a},1)$, $\diag(\zeta_{2^{b}an},\zeta_{2^{b}an}^{e},1)$ for $a,b,e,n$ positive integers with $e\cong\pm 1\pmod{q}$ for each prime power $q\mid 2^{b}n$ and $n$ odd.
	
	If $a\neq 1$ then $2^{b}an\mid d$, while if $a=1$ then $2^{b+1}n\mid d(d-2)$.
\end{proposition}

\begin{proof}
	Let $L_{i}$ be the line $x_{i}=0$ for $i=1,2,3$, the points of $\PP^{2}$ with non-trivial stabilizer are contained in $L_{1},L_{2},L_{3}$, see \S\ref{sect:2eq}. If $a\neq 1$, the three intersection points of these lines have non-cyclic stabilizers and hence are not in $C$. Since the orbits of $L_{2}$ different from $(1:0:0)$, $(0:0:1)$ have cardinality equal to $a2^{b}n$, we get that $a2^{b}n\mid d=\deg(C\cap L_{2})$. Assume $a=1$. We want to prove that $2^{b+1}n\mid d(d-2)$; equivalently, $d$ is even and every prime power $q\mid 2^{b}n$ divides either $d$ or $d-2$.
	
	The point $(0:0:1)$ is distinguished and descends to a rational point of the compression $\bP_{C}$. Let $\bC\subset\bP_{C}$ be the coarse moduli space of the universal curve $\cC\subset\cP_{C}$, by the Lang--Nishimura theorem for stacks this lifts to a rational map $\cC\dashrightarrow\cP_{C}$. If $(0:0:1)\in C$, then the corresponding rational point on $\bP$ is contained in $\bC$, hence $\cP_{C}(k)\neq\emptyset$ by the Lang--Nishimura theorem for tame stacks. This implies that $C$ is defined by a homogeneous polynomial with coefficients in $k$, giving a contradiction. We may thus assume that $(0:0:1)\not\in C$.
	
	The quotient $L_{3}/\aut(C)$ descends to a non-trivial Brauer--Severi curve $\bL_{3}\subset\bP_{C}$, which induces a rational map $\bL_{3}\dashrightarrow\cP_{C}$ by the Lang--Nishimura theorem for stacks. If $\bL_{3}\simeq\PP^{1}$, then $\bP_{C}(k)\neq\emptyset$ again by the Lang--Nishimura theorem for stacks, hence we may assume $\bL_{3}(k)=\emptyset$. Notice that $\diag(-1,-1,1)\in\aut(C)$ fixes $L_{3}$, hence $C$ has normal crossing with $L_{3}$. This implies that $d=\deg (C\cap L_{3})$ is even, since $(C\cap L_{3})/\aut(C)$ descends to a divisor of $\bL_{3}$. 
	
	Fix $q\mid 2^{b}n$ a prime power. There are two cases: either $e\cong 1\pmod{q}$, or $e\cong -1\pmod{q}$. If $e\cong 1\pmod{q}$, then a generic line $L$ containing $(0:0:1)$ satisfies $L\cap L_{3}\cap C=\emptyset$, is stabilized by $\diag(\zeta_{q},\zeta_{q},1)\in\aut(C)$ and the $\left<\diag(\zeta_{q},\zeta_{q},1)\right>$-orbits of $L\cap C$ have cardinality $q$. It follows that $q\mid\deg(L\cap C)=d$.
	
	Assume $e\cong -1\pmod{q}$ (and $q\neq 2$, otherwise $e\cong 1$ too), then $\{(1:0:0),(0:1:0)\}$ is distinguished and descends to a point of $\bL_{3}$ with residue field of degree $2$. This point either is in $\bC$ or not, hence $C$ either contains both $(1:0:0)$ and $(0:1:0)$ or none. If $C$ does not contain them, the orbits of $L_{2}\cap C$ have cardinality $2^{b}n$, hence $2^{b}n\mid\deg (L_{2}\cap C)=d$. Assume that $C$ contains both.
	
	If $q$ is odd, since $e\cong -1\pmod{q}$ the degrees of the orbits of $L_{3}$ different from $(1:0:0)$, $(0:1:0)$ are multiples of $q$, hence $q$ divides $\deg(L_{3}\cap C)-2=d-2$. If $q=2^{b}$ is even, said orbits have degree multiple of $2^{b-1}$. Furthermore, the fact that $\bL_{3}$ is a non-trivial Brauer-Severi variety implies that there is an even number of orbits in $C\cap L_{3}$. This implies that $2^{b}\mid \deg(C\cap L_{3})-2=d-2$.
\end{proof}

\begin{theorem}\label{theorem:no3}
	Let $k$ be a field of characteristic $0$ with algebraic closure $K$, $C\subset\PP^{2}_{K}$ a smooth plane curve over $K$ of degree $d$ prime with $3$ and with field of moduli $k$.
	
	If $C$ is not defined by a homogeneous polynomial with coefficients in $k$, then $\aut(C)$ has the form $C_{a}\times C_{2an}$ generated by $\diag(\zeta_{a},1,1)$, $\diag(1,\zeta_{a},1)$, $\diag(\zeta_{2an},\zeta_{2an}^{e},1)$ for $a,e,n$ positive integers with $e\cong\pm 1\pmod{q}$ for each prime power $q\mid 2n$. Furthermore, if $a\neq 1$ then $2an\mid d$, while if $a=1$ then $4n\mid d(d-2)$.
\end{theorem}

\begin{proof}
	Clearly, we may assume $d\ge 4$. Since $\deg C$ is prime with $3$ and $C$ is not defined by a homogeneous polynomial with coefficients in $k$, by Corollary~\ref{corollary:brauer} the abstract curve $C$ is not defined over $k$ as well.
	
	By Theorem~\ref{theorem:cycles}, $\aut(C)$ is critical. There are four types of critical subgroups of $\PGL_{3}$: we want to check that three are not possible, while the fourth is settled by Proposition~\ref{proposition:crit2}.
	
	\subsection{$\boldsymbol{C_{a}\times C_{an}}$, $\boldsymbol{3|an}$}
	
	In this case $3\mid \deg C$ by Proposition~\ref{proposition:crit1}, which contradicts our hypothesis.

	\subsection{\boldsymbol{$H_{2}, H_{3}$}}
	Recall that the degrees of $H_{2}$, $H_{3}$ are $18,36$ respectively. Observe that, since $3\nmid d$, then $3$ divides the genus $(d-1)(d-2)/2$ of $C$. Furthermore, $3$ divides $|\aut(C)|$. By Riemann-Hurwitz we have
	\[-2\cong \deg R\pmod{3},\]
	where $R\subset C$ is the ramification divisor of $C\to C/\aut(C)$. Let us show that $\deg R$ is a multiple of $3$, this gives the desired contradiction. 
	
	Let $p\in C$ be a point, since the stabilizer acts faithfully on the tangent space of $p$ then it is cyclic of order $m\ge 1$. There are no elements of order $9$ in $H_{3}$, it follows that $9\nmid m$ and hence $|\aut(C)|/m$ is a multiple of $3$, i.e. the degree of every orbit is a multiple of $3$. This clearly implies that the degree of the ramification divisor is a multiple of $3$.
\end{proof}

It might happen that some groups of the form described in Theorem~\ref{theorem:no3} are not the groups of automorphisms of any plane curve of degree $d$. For instance, if $d=4$, the groups of this form are $C_{2}$, $C_{4}$ and $C_{2}\times C_{4}$, but only $C_{2}$ can actually be the group automorphism of a plane quartic \cite{bars}. As a consequence, the result by M. Artebani and S. Quispe \cite{artebani-quispe} holds for every field of characteristic $0$.

\begin{corollary}\label{corollary:quartics}
	Let $k$ be a field of characteristic $0$ with algebraic closure $K$ and $C\subset\PP^{2}_{K}$ a smooth plane quartic over $K$ with field of moduli $k$. If $\aut(C)\neq C_{2}$, then $C$ is defined by a homogeneous polynomial with coefficients in $k$.
\end{corollary}

Since a smooth, projective curve of genus $3$ is either a plane curve or an hyperelliptic curve, and B. Huggins has completely analyzed the problem for hyperelliptic curves \cite{huggins}, this essentially completes the study of fields of moduli of genus $3$ curves in characteristic $0$.

\begin{theorem}\label{theorem:odd}
	Let $C$ be a smooth, plane curve over $K$ of odd degree.
	
	If $C$ has no models over its field of moduli, then the group $\aut(C)=\aut(\PP^{2},C)\subset\PGL_{3}$ is conjugate to the group $C_{a}\times C_{an}$ generated by $\diag(\zeta_{a},1,1)$, $\diag(1,\zeta_{a},1)$, $\diag(\zeta_{an},\zeta_{an}^{e},1)$ for some positive integers $a,e,n$ with $3\mid an$ and $e^{2}-e+1\cong 0\pmod{n}$. Such an $e$ exists if and only if $n$ is odd and $-3$ is a square modulo $n$.

	Furthermore, one of the following holds.
	\begin{itemize}
		\item $an\mid d$, or
		\item $a=1$, $n\mid d^{2}-3d+3$ and $n\neq d^{2}-3d+3$.
	\end{itemize}
\end{theorem}

\begin{proof}
	The genus of $C$ is $(d-1)(d-2)/2$, write $h$ for the genus of $C/\aut(C)$. 
	
	Clearly, we may assume $d\ge 4$. By Theorem~\ref{theorem:cycles}, $\aut(C)$ is critical. There are four types of critical subgroups of $\PGL_{3}$: we want to check that three are not possible, while the fourth is settled by Proposition~\ref{proposition:crit1}.

	\subsection{$\boldsymbol{C_{a}\times C_{a2^{b}n}}$}

		By Proposition~\ref{proposition:crit2}, in this case $\deg C$ is even, which contradicts our hypothesis.
		
	\subsection{$\boldsymbol{H_{2}}$}
		Let $\cG$ be the residual gerbe, thanks to Lemma~\ref{lemma:index} and Corollary~\ref{corollary:indexneutral} it is enough to show that the index of $\cG$ is $1$ cf. Appendix~\ref{sect:index}. Let us start by showing that it divides $4$.
		
		Each subgroup of $H_{2}$ of order $3$ has $3$ fixed points, see the matrices $M_{0}$, $M_{1}$ in \S\ref{sect:hessian}. There are $4$ such subgroups, and the fixed loci of these subgroups are pairwise disjoint (this follows easily from the presentation given in \S\ref{sect:hessian}); we call these the four special triangles. The elements of order $2$ in $H_{2}$ are all conjugate. Each of them fixes one point of each special triangle, and swaps the other two points (see the matrix $M_{2}$ in \S\ref{sect:hessian}). Hence, the stabilizer of each of the $12$ points of the $4$ special triangles is isomorphic to $S_{3}$, and each special triangle is an orbit for the action of $H_{2}$.
		
		These $12$ points form a distinguished subset, hence they descend to a $0$-cycle $Z$ of degree $4$ in the compression $\bP$ of $(\PP^{2},C)$. Since $S_{3}$ is of type $\rR_{2}$ \cite{giulio-tqs2}, the singularities of the points of $Z$ are liftable, hence we get a map $Z\to\cG$ thanks to the Lang--Nishimura theorem for stacks \cite[Theorem 4.1]{giulio-angelo-valuative}. It follows that the index of $\cG$ divides $4$. To conclude, it is enough to show that the index is odd.
		
		Consider the $9$ lines fixed by the $9$ elements of order $2$ of $H_{2}$. Any intersection point of two of these lines has non-cyclic stabilizer, since it is fixed by two elements of order $2$. In particular, these intersection points are not contained in $C$. By what we have said above, we can easily see that the intersection points of the lines are precisely the $12$ points of the special triangles. It follows that the intersection of $C$ with these lines has exactly $9d$ points divided in $9d/9=d$ orbits. These form a distinguished subset, hence they descend to a divisor of degree $d$ on the compression $\bC$ of $C$. Since we have a rational map $\bC\dashrightarrow\cG$, by the Lang--Nishimura theorem for stacks \cite[Theorem 4.1]{giulio-angelo-valuative} we get that the index of $\cG$ divides $d$, and hence it is odd.

	\subsection{$\boldsymbol{H_{3}}$}

		We already know the fixed loci of the elements of $H_{2}\subset H_{3}$. The $18$ elements of $H_{3}\setminus H_{2}$ all have order $4$, and each fixes $3$ points. The $9$ cyclic subgroups of order $4$ are all conjugate.
		
		An element $g$ of order $4$ fixes $3$ points, and none of them is also fixed by an element of order $3$ (see the matrix $M_{3}$ in \S\ref{sect:hessian}). For $2$ of these $3$ fixed points, $g$ acts with eigenvalues $(\pm i,-1)$ on the tangent space, while on the third point the eigenvalues are $(\pm i,\mp i)$. Let us call $4$-points the ones stabilized by an element of order $4$ which acts with eigenvalues $(\pm i,-1)$ on the tangent space. 
		
		Let $L$ be the line fixed by $g^{2}$: it contains two $4$-points fixed by $g$, while all the other $<g>$-orbits in $L$ have order $2$. Since $L$ has transversal intersection with $C$ and $|C\cap L|=d$ is odd, we get that $C$ contains exactly one of the two $4$-points contained in $L$. Since there are $9$ elements of order $2$, we get that $C$ contains $9$ $4$-points which form a unique $H_{3}$-orbit, since the cyclic subgroups of order $4$ are conjugate.
		
		These $9$ points thus form a distinguished subset and they descend to a rational point in the compression $\bC\subset\bP$ of $C$. By the Lang--Nishimura theorem for stacks \cite[Theorem 4.1]{giulio-angelo-valuative} applied to $\bC\dashrightarrow \cG$, we conclude that $\cG(k)\neq\emptyset$, which is a contradiction.

\end{proof}

Our method works in the case in which $d$ is a multiple of $6$, too, but the results are less clear-cut. By cross-checking Theorem~\ref{theorem:cycles} with the list of possible groups of automorphisms of plane sextics \cite{badr-bars6}, we obtain the following result for curves of degree $6$.

\begin{theorem}\label{theorem:sextics}
	Let $k$ be a field of characteristic $0$ with algebraic closure $K$ and $C\subset\PP^{2}_{K}$ a smooth plane sextic over $K$ with field of moduli $k$. If $\aut(C)$ is not isomorphic to one of the groups
	\[C_{2},~C_{3},~C_{4},~C_{6},~C_{3}^{2},~H_{2}=C_{3}^{2}\rtimes C_{2},~H_{3}=C_{3}^{2}\rtimes C_{4},\]
	then there exists a model $\fC$ of $C$ over $k$ with an embedding $\fC\subset P_{\fC}$ in a Brauer-Severi surface $P_{\fC}$ over $k$. Furthermore, if $\aut(C)$ is also not trivial and not isomorphic to $C_{2}^{2}$, we may choose $\fC$ so that $P_{\fC}\simeq\PP^{2}_{k}$.
	
	Moreover, if $\sqrt{3},\sqrt{-1}\in k$ and $\aut(C)\simeq C_{3}^{2}\rtimes C_{4}$, there exists a model $\fC$ over $k$ with and embedding $\fC\subset\PP^{2}_{k}$. 
\end{theorem}

Just as Theorem~\ref{theorem:cycles}, Theorem~\ref{theorem:six} is a particular case of a result about algebraic structures.

\begin{theorem}\label{theorem:structures}
	Let $\xi$ be an algebraic structure on $\PP^{2}_{K}$. There exists a finite extension $k'/k_{\xi}$ with $[k':k_{\xi}]\le 3$ such that $\xi$ descends to a structure on $\PP^{2}_{k'}$. If $\xi$ is the structure given by a smooth plane curve $C$, we can choose $k'$ so that $[k':k_{\xi}]=[k':k_{C}]\mid\deg C$.
\end{theorem}

\begin{proof}
	Up to base change, we may assume $k_{\xi}=k$.
	
	Assume first that $(\PP^{2}_{K},\xi)$ has a model $(P,\xi_{0})$ over $k$ with $P$ a Brauer--Severi surface. By \cite[Proposition 4.5.4]{gille-szamuely}, there exists a finite extension $k'/k$ of degree $3$ with $P_{k'}\simeq\PP^{2}_{k'}$, hence $(P,\xi_{0})_{k'}$ is the desired model over $k'$. If $\xi$ is the structure given by a smooth plane curve $C$, $\fC\subset P$ is the corresponding model and $3=[k':k]$ does not divide $\deg C$, the Brauer--Severi surface $P$ contains a divisor $\fC$ of degree prime with $3$ and hence it has index $1$, i.e. $P\simeq\PP^{2}_{k}$. We may then choose $k'=k$.
	
	Assume now that $\xi$ is not defined over $k$. Again, $\aut(\xi)$ is critical by Theorem~\ref{theorem:critical}. Let us analyze each case.
	
	\subsection{$\boldsymbol{C_{a}\times C_{an}}$, $\boldsymbol{3|an}$}
		Using the same notation as in the proof of Theorem~\ref{theorem:odd}, there exist a point $z$ in the compression $\bP$ such that $[k(z):k]=3$ and the $3$ corresponding points in $\PP^{2}_{K}/\aut(C)$ are singularities of type $\rR$. By the Lang--Nishimura theorem for stacks \cite[Theorem 4.1]{giulio-angelo-valuative} applied to $\bP\dashrightarrow \cP$, we get $\cP(k(z))\neq\emptyset$. This implies that $\xi$ descends to a structure on $\PP^{2}_{k(z)}$.
		
		Assume that $\xi$ is the structure given by a smooth plane curve $C$. Since $C$ has no models over $k$, then $3=[k(z):k]\mid \deg C$ by Proposition~\ref{proposition:crit1}. 
	
	\subsection{$\boldsymbol{C_{a}\times C_{a2^{b}n}}$}
		The line $(x:y:0)$ is distinguished, hence it descends to a Brauer--Severi curve $B$ in the compression. Hence, there exists a point $b\in B$ which is smooth in $\bP$ and such that $[k(b):k]=2$. By the Lang--Nishimura theorem for stacks \cite[Theorem 4.1]{giulio-angelo-valuative} applied to $\bP\dashrightarrow \cP$, we get $\cP(k(b))\neq\emptyset$. This implies that $C$ descends to a curve $\fC$ over $k(b)$ with an embedding $\fC\subset\PP^{2}_{k(b)}$.
		
		Assume that $\xi$ is the structure given by a smooth plane curve $C$. Since $C$ has no models over $k$, then $2=[k(b):k]\mid \deg C$ by Proposition~\ref{proposition:crit2}. 
		
	\subsection{\boldsymbol{$H_{2}, H_{3}$}}
		There are exactly $9$ elements of order $2$ in $\aut(C)$ and they are all conjugate. Each of them fixes a different line in $\PP^{2}$, these $9$ lines form a distinguished subset. Since the $9$ elements of order $2$ are conjugate, $\aut(C)$ acts transitively on the set of these $9$ lines, hence they descend to a Brauer--Severi curve $B$ in the compression. We conclude as in the previous case.
		
		Assume now that $\xi$ is the structure given by a smooth plane curve $C$. Since $C$ has no models over $k$ and $\aut(\xi)$ is either $H_{2}$ or $H_{3}$, then $2\mid \deg C$ by Theorem~\ref{theorem:odd}.
\end{proof}

\appendix

\section{The index of a gerbe}\label{sect:index}

The index of a gerbe $\cG$ is the greatest common divisor of the degrees of finite splitting fields of $\cG$. In this appendix, we prove some abstract facts about the index of a finite étale gerbe that we used in the proof of Theorem~\ref{theorem:odd}. The main results are Lemma~\ref{lemma:index} and Corollary~\ref{corollary:indexneutral}. We think these facts will be useful in future works, so we have made an effort to prove general statements, whereas a much simpler argument would have been sufficient for the application in Theorem~\ref{theorem:odd}.

Let $\cG$ be a finite étale gerbe over a field $k$ of arbitrary characteristic, write $k^{s}/k$ for a separable closure. By \cite[Lemma 4.5]{giulio-angelo-moduli}, there exists a section $b:\spec k^{s}\to\cG$. Write $G=\pi_{1}(\cG_{k^{s}},b)$, it is a finite group. We have a short exact sequence
\[1\to G \to \pi_{1}(\cG,b) \to \gal(k^{s}/k) \to 1.\]

\begin{lemma}\label{lemma:gerbequotient}
	The gerbe $\cG$ is isomorphic to the quotient stack $[\spec k^{s}/\pi_{1}(\cG,b)]$, where $\pi_{1}(\cG,b)$ acts on $\spec k^{s}$ with the projection $\pi_{1}(\cG,b)\to\gal(k^{s}/k)$.
\end{lemma}

\begin{proof}
	This is a direct consequence of the fact that the morphism $b:\spec k^{s}\to\cG$ is a universal cover which is a Galois $\pi_{1}(\cG,b)$-cover.
\end{proof}

If we have a rational point $p:\spec k\to\cG$, by functoriality we get a section $\gal(k^{s}/k)\to\pi_{1}(\cG,p)$. An étale path on $\cG_{k^{s}}$ from $p$ to $b$ defines an isomorphism $\pi_{1}(\cG,p)\simeq\pi_{1}(\cG,b)$ which commutes with the projection to $\gal(k^{s}/k)$, hence we get a section $\gal(k^{s}/k)\to\pi_{1}(\cG,b)$ well defined up to conjugation by elements of $G$ (which identifies with the group of étale paths from $b$ to itself in $\cG_{k^{s}}$). Denote by $\cS_{\cG/k}$ the set of sections $\gal(k^{s}/k)\to\pi_{1}(\cG,b)$ modulo the action of $G$ by conjugation, we have constructed a map
\[\cG(k)\to\cS_{\cG/k}.\]

Notice that $\cG(k)$ is a groupoid, whereas $\cS_{\cG/k}$ is a set. The following lemma is well-known, see e.g. \cite[Proposition 9.3]{niels-angelo15}, though the proofs available in the literature are quite technical. Let us give an elementary proof.

\begin{lemma}\label{lemma:gerbesect}
	The map $\cG(k)\to\cS_{\cG/k}$ identifies $\cS_{\cG/k}$ with the set of isomorphism classes of $\cG(k)$.
\end{lemma}

\begin{proof}
	Consider two sections $p,q:\spec k\to\cG$ with equal image in $\cS_{\cG/k}$ and write $p^{s},q^{s}$ for the two compositions $\spec k^{s}\to\spec k\to\cG$; equivalently, there is an étale path $\gamma$ from $p^{s}$ to $q^{s}$ in $\cG_{k^{s}}$ such that the natural map $\gal(k^{s}/k)\xrightarrow{q_{*}}\pi_{1}(\cG,q)$ is equal to the composition $\gal(k^{s}/k)\xrightarrow{p_{*}}\pi_{1}(\cG,q)\xrightarrow{\gamma^{*}}\pi_{1}(\cG,p)$. This means that the two natural actions of $\gal(k^{s}/k)$ on the two fiber functors associated with $p^{s}$ and $q^{s}$ are identified by $\gamma$.
	
	Consider the étale cover $p:\spec k\to \cG$, its fiber over $p^{s}$ has a preferred object which $\gamma$ maps to a point in the fiber over $q^{s}$, which by construction is the set of isomorphisms $p^{s}\simeq q^{s}$; we have thus constructed a preferred isomorphism $p^{s}\simeq q^{s}$. Now consider the fibered product $E=\spec k\times_{\cG}\spec k$, where the two maps $\spec k\to\cG$ are $p$ and $q$, it is a finite étale scheme over $k$ which represents isomorphisms between $p$ and $q$; in particular, we have a preferred point $\spec k^{s}\to E$. The fact two actions of $\gal(k^{s}/k)$ on the two fiber functors associated with $p^{s}$ and $q^{s}$ are identified by $\gamma$ implies that the preferred point $\spec k^{s}\to E$ is Galois invariant, hence it is $k$-rational and we get an isomorphism $p\simeq q$.
	
	Consider now a section $s:\gal(k^{s}/k)\to\pi_{1}(\cG,b)$. Notice that $\spec k^{s}$ is a $\gal(k^{s}/k)$-torsor over $k$, define $T/k$ as the induced $\pi_{1}(\cG,b)$-torsor
	\[\spec k^{s}\times^{\gal(k^{s}/k)}\pi_{1}(\cG,b)=(\spec k^{s}\times\pi_{1}(\cG,b))/\gal(k^{s}/k)\]
	using $s:\gal(k^{s}/k)\to\pi_{1}(\cG,b)$. Since $s$ is a section of $\pi_{1}(\cG,b)\to\gal(k^{s}/k)$, the induced torsor $T\times^{\pi_{1}(\cG,b)}\gal(k^{s}/k)$ is again $\spec k^{s}$, hence we get an equivariant map $T\to\spec k^{s}$. By definition of quotient stack, this gives a point $p:\spec k\to[\spec k^{s}/\pi_{1}(\cG,b)]=\cG$. It is straightforward to check that the Galois section associated with $p$ is equivalent to $s$.
\end{proof}

Because of Lemma~\ref{lemma:gerbesect}, finding a rational point of $\cG$ is equivalent to splitting the associated short exact sequence of fundamental groups.

Recall that the degree of $\cG$ is the degree of the automorphism group of any geometric point.

\begin{lemma}\label{lemma:index}
	Let $\cG$ be a finite étale gerbe over $k$. The prime factors of the index of $\cG$ divide the degree of $\cG$. 
\end{lemma}

\begin{proof}
	We have that $\cG_{k^{s}}\simeq\cB_{k^{s}}G$ with $G=\pi_{1}(\cG_{k^{s}})$. Since $G$ is finite, there exists a normal, finite index subgroup of $\pi_{1}(\cG)$ with trivial intersection with $G\subset\pi_{1}(\cG)$. Because of this, there are finite groups $H$, $Q$ with a commutative diagram of short exact sequences
	\[\begin{tikzcd}
		1\rar	&	G\rar\dar[equal]	&	\pi_{1}(\cG)\rar\dar	&	\gal(k^{s}/k)\rar\dar	&	1	\\
		1\rar	&	G\rar				&	H\rar					&	Q\rar					&	1	
	\end{tikzcd}\]
	where the vertical arrows are surjective. In particular, $\pi_{1}(\cG)$ identifies with the fibered product $H\times_{Q}\gal(k^{s}/k)$.
	
	Let $p$ be a prime which does not divide $|G|$ and choose $P\subset H$ a $p$-Sylow subgroup of $G$; then $P\hookrightarrow Q$ is injective and defines a $p$-Sylow of $Q$. Let $k'\subset k^{s}$ be the fixed field of the inverse image of $P$ in $\gal(k^{s}/k)$, we get an induced map $\gal(k^{s}/k')\to\pi_{1}(\cG)$. By Lemma~\ref{lemma:gerbesect}, $\cG(k')\neq\emptyset$, hence $p$ does not divide the index of $\cG$ since by construction $p\nmid [k':k]=[Q:P]$.
\end{proof}

\begin{remark}\label{remark:index}
	If $G$ is abelian, there is a natural action of $\gal(k^{s}/k)$ on $G$ and $\cG$ corresponds to an element $g$ of $\H^{2}(\gal(k^{s}/k),G)$, see for instance \cite[Chapitre IV, \S 3.4]{giraud}, and Lemma~\ref{lemma:index} follows from the fact that the order of $g$ divides the index of $\cG$ (this can be seen by taking restriction and corestriction along the splitting fields of $\cG$). The case in which $G$ is solvable then follows by an easy induction argument.
\end{remark}

The following result is due to L. A. \v{S}emetkov and B. Sambale, see \cite{semetkov} \cite[Theorems 2, 3 and 9]{sambale}.

\begin{theorem}[{\v{S}emetkov, Sambale}]\label{theorem:sambale}
	Consider a short exact sequence of finite groups
	\[1\to G\to H \to Q \to 1.\]
	For every prime $p$ dividing $|G|$, assume that there exists a $p$-Sylow subgroup $P\subset Q$ with a section $P\to H$. Furthermore, assume either that
	\begin{itemize}
		\item the Sylow subgroups of $G$ are abelian, or
		\item $[G,G]$ is abelian and intersects trivially the center of $G$.
	\end{itemize}
	Then there is a section $Q\to H$.
\end{theorem}

The case in which the Sylow subgroups of $G$ are abelian is a direct consequence of \v{S}emetkov's theorem \cite[Theorem 3]{sambale}. The second case is essentially \cite[Theorem 9]{sambale} with a weakened hypothesis; it is straightforward to check that the proof still works with the weakened hypothesis.

\begin{corollary}\label{corollary:indexneutral}
	Let $\cG$ be a finite étale gerbe over $k$, write $G$ for the automorphism group of a geometric point. Assume either that
	\begin{itemize}
		\item the Sylow subgroups of $G$ are abelian, or
		\item $[G,G]$ is abelian and intersects trivially the center of $G$.
	\end{itemize}
	The index of $\cG$ is $1$ if and only if $\cG$ is neutral.
\end{corollary}

\begin{proof}
	The ``if'' part is obvious, assume that the index is $1$. As in the proof of Lemma~\ref{lemma:index}, there exists a commutative diagram of short exact sequences
	\[\begin{tikzcd}
		1\rar	&	G\rar\dar[equal]	&	\pi_{1}(\cG)\rar\dar	&	\gal(k^{s}/k)\rar\dar	&	1	\\
		1\rar	&	G\rar				&	H\rar					&	Q\rar					&	1	
	\end{tikzcd}\]
	with $H, Q$ finite, and we get an identification $\pi_{1}(\cG)\simeq H\times_{Q}\gal(k^{s}/k)$.
	
	Since $\cG$ has index $1$, for every prime $p$ dividing $|G|$ there exists a finite extension $k'/k$ with $p \nmid [k':k]$ and a section $\gal(k'^{s}/k')\to\pi_{1}(\cG)$. While $k'/k$ might be not separable, the natural map $\gal(k'^{s}/k')=\aut(\bar{k}/k')\to\gal(k^{s}/k)=\aut(\bar{k}/k)$ is injective. This implies that the image of $\gal(k'^{s}/k')\to\pi_{1}(\cG)$ does not intersect $G$. 
	
	Hence, up to replacing $Q$ with a larger quotient of $\gal(k^{s}/k)$, we can assume that the image of the composition $\gal(k'^{s}/k')\to \pi_{1}(\cG)\to H$ maps injectively in $Q$. Furthermore, we can assume that this holds for every prime $p$ dividing $|G|$. Since $p\nmid[k':k]$, the image of $\gal(k'^{s}/k')\to Q$ contains a $p$-Sylow of $Q$, hence the hypothesis of Theorem~\ref{theorem:sambale} is satisfied. We thus get a section $Q\to H$ which induces a section $\gal(k^{s}/k)\to\pi_{1}(\cG)$. We conclude by Lemma~\ref{lemma:gerbesect}.
\end{proof}

\begin{remark}
	If $G$ is abelian, Corollary~\ref{corollary:indexneutral} can be proved by an easy argument with cohomology similarly to Remark~\ref{remark:index}. The case in which $G$ is an iterated semi-direct product of abelian groups with pairwise coprime degrees follows by an easy induction argument, and this is sufficient for the present article. Still, for future applications it is better to have the more general result, and as far as we know this requires using group theory and the theorems of \v{S}emetkov and Sambale.
\end{remark}

\section{The intersection of the fields of definition}

There is some confusion in the literature about the relation between the field of moduli and the intersection of the fields of definition. In the context of fields of moduli of Galois coverings \cite[Proposition 2.7]{coombes-harbater}, K. Coombes and D. Harbater introduced the idea of using the Artin--Schreier theorem to prove that, if the base field is $\QQ$, the field of moduli is the intersection of the fields of definitions. 

The aim of this brief appendix is to clarify the relation between the two concepts in general. Denote by $\zeta_{n}$ a primitive $n$-th root of the unity, and by $\QQ(\zeta_{\infty})$ the field generated over $\QQ$ by all roots of the unity.

\begin{lemma}\label{lemma:galoisgen}
	Let $k$ be any field with separable closure $k^{s}$. The following are equivalent.
	
	\begin{enumerate}[(i)]
		\item The group $\gal(k^{s}/k)$ is not topologically generated by closed subgroups of the form $\ZZ_{p}$ for some $p$.
		\item The group $\gal(k^{s}/k)$ is not topologically generated by closed subgroups of cohomological dimension $\le 1$.
		\item The group $\gal(k^{s}/k)$ has the form $A\rtimes C_{2}$, where $A$ is abelian and $C_{2}$ acts by $-1$ on $A$.
		\item The characteristic is $0$, $i\not\in k$, $k$ contains a copy of $\QQ(\zeta_{\infty})\cap\RR$ and $\gal(\bar{k}/k(i))$ is abelian.
	\end{enumerate}
\end{lemma}

\begin{proof}
	$(iv)\Rightarrow (iii)$ Assuming $(iv)$, we have that $\gal(\bar{k}/k)$ is an extension of $C_{2}$ by an abelian pro-finite group $A=\gal(\bar{k}/k(i))$. We have to show that the extension is split and that $C_{2}$ acts by $-1$ on $A$. Let $\sigma\in\gal(\bar{k}/k)$ be any element such that $\sigma(i)=-i$; since $k$ contains $\QQ(\zeta_{\infty})\cap\RR$, then $\sigma$ restricts to the only non trivial element of $\gal(\QQ(\zeta_{\infty})/\QQ(\zeta_{\infty})\cap\RR)$, i.e. complex conjugation.
	
	 Since $\cha k=0$, $k(i)$ contains all roots of unity and $\gal(\bar{k}/k(i))$ is abelian, by Kummer theory $\bar{k}$ is generated over $k(i)$ by elements $a$ such that $a^{n}\in k(i)$ for some $n$. Consider such an element $a\in \bar{k}$, $g$ any element of $\gal(\bar{k}/k(i))$ and $\zeta_{n}\in k(i)$ a primitive $n$th root of $1$. Since $\sigma(a)^{n}=\sigma(a^{n})\in k(i)$, there exists some $d$ such that $g\sigma(a)=\zeta^{d}_{n}\sigma(a)$. Hence,
	\[\sigma^{-1}g\sigma(a)=\sigma^{-1}\zeta_{n}^{d}\sigma(a)=\zeta_{n}^{-d}a=g^{-1}(a),\]
	which implies that $\sigma^{-1}g\sigma=g^{-1}$, i.e. $C_{2}$ acts by $-1$ on $\gal(\bar{k}/k(i))$. In particular, for $g=\sigma^{2}$ we get the equality $\sigma^{4}=1$. By Artin--Schreier, this implies that $\sigma^{2}=1$ and hence $\gal(\bar{k}/k)\to C_{2}$ is split.
	
	$(iii)\Rightarrow (ii)$ If $(a,1)\in A\rtimes C_{2}$ is an element over the non-trivial element $1$ of $C_{2}$, then $(a,1)^{2}=(a-a,1+1)=(0,0)$, i.e. every element over $1\in C_{2}$ is $2$-torsion. Since a group with non-trivial torsion has infinite cohomological dimension, then $A\rtimes C_{2}$ cannot be generated by closed subgroups of cohomological dimension $\le 1$.
	
	$(ii)\Rightarrow (i)$ The group $\ZZ_{p}$ has cohomological dimension $1$ for every $p$.
	
	$(i)\Rightarrow (iv)$ If the characteristic is not $0$, or if $\cha k=0$ and $i\in k$, by Artin--Schreier $\gal(k^{s}/k)$ is torsion free and hence it is clearly generated by closed subgroups of the form $\ZZ_{p}$ for some $p$. We can thus assume that $\cha k=0$ and $i\not \in k$. By the preceding case, the index $2$ subgroup $\gal(\bar{k}/k(i))$ is generated by closed subgroups of the form $\ZZ_{p}$ for some $p$. Since this is false for $\gal(\bar{k}/k)$, we get that $\gal(\bar{k}/k)$ does not contain a closed subgroup of the form $\ZZ_{2}$ which acts non-trivially on $k(i)$.
	
	Let $\sigma\in\gal(\bar{k}/k)$ be an element such that $\sigma(i)=-i$; I claim that $\sigma^{2}=1$. Let $h\subset \bar{k}$ be the field fixed by $\sigma$. The Galois group $\gal(\bar{k}/h)$ is pro-finite abelian, hence by Artin--Schreier it contains either a copy of $C_{2}$ or a copy of $\ZZ_{2}$ which acts non-trivially on $k(i)$, and as we have seen the latter is not possible. By \cite[Satz 1.13]{geyer}, this implies that $\gal(k/h)\simeq C_{2}$, i.e. $\sigma\in \gal(k/h)$ is the only non-trivial element and it has order $2$.
	
	If $g$ is an element of $\gal(\bar{k}/k(i))$, then $\sigma g$ has order $2$ by the above, equivalently $\sigma g \sigma^{-1}= g^{-1}$. This implies that $g\mapsto g^{-1}$ is an automorphism of $\gal(\bar{k}/k(i))$, i.e. $A=\gal(\bar{k}/k(i))$ is abelian.
	
	Since $\sigma^{2}=1$, the fixed field $L$ is real closed and hence contains a copy of $\QQ(\zeta_{\infty})\cap\RR$. By adding $i$, we get an embedding $\QQ(\zeta_{\infty})\subset \bar{k}$ such that $\sigma$ acts by complex conjugation on $\QQ(\zeta_{\infty})$. It remains to show that $\QQ(\zeta_{\infty})\cap\RR\subset L$ is contained in $k$, or equivalently that the image $G$ of $\gal(\bar{k}/k)\to\gal(\QQ(\zeta_{\infty})/\QQ)$ is generated by complex conjugation.
	
	Notice that the abelianization of $A\rtimes C_{2}$ is $A/2\times C_{2}$, which is $2$-torsion. Since $\gal(\QQ(\zeta_{\infty})/\QQ)$ is abelian, it follows that $G$ is an abelian group whose elements have order $2$. If the image $G$ is not generated by complex conjugation, there exists an element $\omega\in G$ different from complex conjugation and with $\omega(i)=-i$. 
	
	A torsion element of $\gal(\QQ(\zeta_{\infty})/\QQ)$ is determined by its images in $\gal(\QQ(\zeta_{p})/\QQ)$ for $p$ odd prime plus the image in $\gal(\QQ(i)/\QQ)$, since $\gal(\QQ(\zeta_{p},\zeta_{p^{2}},\dots)/\QQ(\zeta_{p}))$ and $\gal(\QQ(i=\zeta_{4},\zeta_{8},\dots)/\QQ(i))$ are torsion free. Since $\omega$ is $2$-torsion and $\gal(\QQ(\zeta_{p})/\QQ)$ is cyclic, then $\omega(\zeta_{p})$ is either $\zeta_{p}$ or $\zeta_{p}^{-1}$. Since $\bar{\zeta_{p}}=\zeta_{p}^{-1}$, $\omega(i)=\bar{i}=-i$ and $\omega$ is different from complex conjugation, it follows that there exists at least one odd prime $p$ with $\omega(\zeta_{p})=\zeta_{p}$.
	
	Let $\omega'\in\gal(\bar{k}/k)$ be any lifting of $\omega$. Since $\omega'(i)=-i$, then $\omega'\in \gal(\bar{k}/k)=A\rtimes C_{2}$ has the form $(a,1)$ for some $a\in A$, hence $\omega'^{2}=(a,1)^{2}=(a-a,1+1)=(0,0)$. It follows that the field fixed by $\omega'$ is real closed and contains $\zeta_{p}$, which is absurd. This completes the proof of the equivalence of the conditions.
\end{proof}

Examples of fields satisfying the conditions of Lemma~\ref{lemma:galoisgen} are $\RR$ and $\RR((t))$.

The following is a direct consequence of Lemma~\ref{lemma:galoisgen} and \cite[Proposition 4.3]{giulio-angelo-moduli}. Recall that a group is virtually abelian if it has a finite index abelian subgroup.

\begin{proposition}\label{proposition:fieldint}
	Let $k$ be a field with algebraic closure $\bar{k}$. Let $(X,\xi)$ be an algebraic space of finite type over $\bar{k}$ with a structure in the sense of \cite{giulio-angelo-moduli}, assume that $\underaut_{\bar{k}}(\xi)$ is finite and reduced and denote by $k_{\xi}$ the field of moduli. 
	
	If the intersection of the fields of definition of $(X,\xi)$ is not $k_{\xi}$, then it is $k_{\xi}(i)$ and $k_{\xi}$ satisfies the equivalent conditions of Lemma~\ref{lemma:galoisgen}. In particular, $\cha k=0$ and $\gal(\bar{k}/k)$ is virtually abelian.
\end{proposition}

\begin{corollary}\label{corollary:exception}
	Assume that $\underaut(X,\xi)$ is finite and reduced. If any of the following assumptions holds, the field of moduli is the intersection of the fields of definitions.
	\begin{itemize}
		\item $\cha k\neq 0$.
        \item $\cha k= 0$ and $\gal(\bar{k}/k)$ is not virtually abelian.
		\item $-1$ is a square in $k$.
	\end{itemize}
\end{corollary}

\bibliographystyle{amsalpha}
\bibliography{p2}

\end{document}